\documentclass[a4paper,reqno]{amsart}
 \RequirePackage{etex}
\usepackage{amssymb,latexsym,amsmath,amscd,amsthm,amsfonts, enumerate}
\usepackage{multirow}
\usepackage{color}
\usepackage[all]{xy}
\usepackage{etex}
\usepackage{caption}
\usepackage{soul}
\usepackage{titletoc}
\usepackage[normalem]{ulem}
\usepackage{graphicx}
\usepackage{amsfonts}
\usepackage{amsmath}

\usepackage{amsthm}
\usepackage{mathrsfs}
\usepackage{color}
\usepackage{tikz}
\usetikzlibrary{intersections} 
\usepackage{mathrsfs,amscd,amssymb,amsthm,amsmath,bm,graphicx,psfrag,subfigure,url}
\usepackage{tikz-cd}
\usepackage{graphicx}
\usepackage{subfigure}
\usepackage{tabularx}
\usepackage{floatrow}
\usepackage{colortbl}
\newtheorem{thmIntro}{Theorem}    
\usetikzlibrary{shapes.geometric}

\newtheorem{theorem}{Theorem}[section]
\newtheorem{lemma}[theorem]{Lemma}
\newtheorem{proposition}[theorem]{Proposition}
\theoremstyle{definition}
\newtheorem{definition}[theorem]{Definition}
\newtheorem{example}[theorem]{Example}

\newtheorem{corollary}[theorem]{Corollary}
\theoremstyle{remark}
\newtheorem{remark}[theorem]{Remark}

\usepackage{tikz,tikz-cd}
\usepackage{mathrsfs}
\usepackage[all]{xy}
\usepackage{tikz}
\usepackage{extarrows}
\usepackage{tikz-cd}
\usetikzlibrary{calc}
\usetikzlibrary{matrix,arrows,decorations.pathmorphing}
\usetikzlibrary{shapes.geometric,positioning}
\usetikzlibrary{arrows,decorations.pathmorphing,decorations.pathreplacing}
\usetikzlibrary{positioning,shapes,shadows,arrows}

\usepackage[colorlinks=true,hyperindex]{hyperref}
\newcommand\myshade{85}
\colorlet{mylinkcolor}{black}
\colorlet{mycitecolor}{red}
\colorlet{myurlcolor}{cyan}

\hypersetup{
  linkcolor  = mylinkcolor!\myshade!black,
  citecolor  = mycitecolor!\myshade!black,
  urlcolor   = myurlcolor!\myshade!black,
  colorlinks = true,
}

\numberwithin{equation}{section}

\definecolor{dark-green}{RGB}{14,150,2}
\definecolor{red}{RGB}{250,0,0}

\UseRawInputEncoding
\begin{document}
\title[Maximal almost pre-rigid representation]
{A geometric realization for maximal almost pre-rigid representations over type $\mathbb{D}$ quivers}

\author{Jianmin Chen}
\address{School of Mathematical Sciences, Xiamen University, 361005, Fujian, PR China.}
\email{chenjianmin@xmu.edu.cn}

\author{Yiting Zheng}
\address{School of Mathematical Sciences, Xiamen University, 361005, Fujian, PR China.}
\email{ytzhengxmu@163.com}

 
\dedicatory{}%
\commby{}%


\begin{abstract}
By using the equivariant theory of group actions, we give a geometric model for the category of  finite dimensional representations over a type   $\mathbb{D}$ quiver $Q_{D}$ with $n$ vertices and directional symmetry. Furthermore, we  introduce the notion of  maximal almost pre-rigid representations  over $Q_{D}$, which form  a  family of objects counted by   the generalized Catalan number. We present    a geometric realization for   maximal almost pre-rigid representations and   prove that  the endomorphism algebras of maximal almost pre-rigid representations are   tilted algebras of type $Q_{\overline{D}}$,  
  where 
$Q_{\overline{D}}$ is a 
quiver obtained by adding $n-2$ new vertices and $n-2$ arrows to the quiver $Q_{D}$. Additionally, 
we define a partial order on
the set of maximal almost pre-rigid representations, which therefore presents  a representation-theoretic interpretation of the    type-$\mathbb{D}$ Cambrian lattice determined by
  $Q_{D}$. Meanwhile, we obtain a  representation-theoretic   interpretation of the    type-$\mathbb{B}$  Cambrian lattices. 
\end{abstract}

\maketitle
\tableofcontents
\section{Introduction}\label{sec.1}

Coxeter groups, a type of group generated by reflection, play an important role in various fields such as algebra, geometry, combinatorics, and theoretical physics.
 The combinatorial properties of these groups, especially their exponents and Coxeter numbers, are closely related to the enumeration of some combinatorial structures. As we know, the exponents and Coxeter number  of a finite irreducible Coxeter group $W$  give rise to the $W$-Catalan numbers \cite{MR2031858,MR1644234,MR1483446}, which coincide with the   Catalan numbers in the case of type-$\mathbb{A}$  Coxeter groups. 
$W$-Catalan numbers arise independently in several areas of mathematics. 
One significant collection of objects counted by the $W$-Catalan number is the set of noncrossing partitions corresponding to $W$ \cite{MR2031858,MR1483446,MR2032983,MR1990581,MR1909925}.
Another class of objects  counted by the $W$-Catalan number is the set of nonnesting partitions associated to $W$. 
These objects  exhibit a remarkable connection with noncrossing partitions, as demonstrated in \cite{MR1483446}. They also appear   in several closely related contexts, including double affine Hecke algebras, two-sided cells, and coinvariant rings \cite{MR2486939}.
   In 2001,  Fomin and   Zelevinsky  introduced the  concept of clusters in the  root system associated with  a Coxeter group $W$, whose count   also coincides with the $W$-Catalan number. 
   
 Cambrian lattices and Coxeter-sortable elements associated with  $W$ are two  classes of objects enumerated by $W$-Catalan numbers, both of which were introduced by Reading. The Cambrian lattices were defined as   lattice quotients of the weak order on $W$ modulo certain congruences, that depend on the  selection of an orientation for the Coxeter diagram \cite{rea06}.
The Hasse diagram for each Cambrian lattice is isomorphic to the 1-skeleton of the generalized associahedron for the corresponding Coxeter group, as detailed in~\cite{rea06,MR2486939,MR2336311,MR2781960}. 
Additionally,  over an algebraically closed field, 
the elements of  Cambrian lattice coming from a Dynkin quiver $Q$ are in bijection with several families of objects:
the clusters of the  cluster algebra whose initial seed is given by $Q$,  
the (isoclasses of) cluster tilting objects in the cluster category  corresponding to $Q$,
the finitely generated torsion classes over the path algebra of $Q$, 
the finitely generated wide subcategories in the category of  finite dimensional representations over $Q$, 
the finitely generated semistable subcategories in the category of finite dimensional representations over $Q$, 
as well as the noncrossing partitions associated with $Q$ \cite{IT}.

In \cite{BGMR}, Barnard-Gunawan-Meehan-Schiffler introduced  a new Catalan object, which is called  a maximal almost rigid representation, to give a representation-theoretic interpretation of the type-$\mathbb{A}$ Cambrian lattices as follows.   Specifically, let $Q$ be a type $\mathbb{A}$ quiver with $n$ vertices.  The authors presented 
a geometric model for the category of  finite dimensional representations over  $Q$ via a   polygon $P(Q)$ with $n+1$ vertices.
Based on this model,
they realized maximal almost rigid representations over $Q$  as triangulations of $P(Q)$.
 Moreover, by defining a partial order on the set of maximal almost rigid representations, they showed that this partial order is a type-$\mathbb{A}$ Cambrian lattice.
Motivated by this   work,  we aim to seek out  a class of representations over type $\mathbb{D}$ quivers  counted by  the generalized Catalan number, and 
give a representation-theoretic interpretation  for the  type-$\mathbb{D}$ Cambrian lattices.

Throughout this paper, we consider the type $\mathbb{D}$ quiver $Q_{D}$ with $n$ vertices and directional symmetry. Let $\Bbbk$ be an algebraically closed field whose characteristic is not 2.
By \cite[Section 2.1]{R}, there  exist  a type $\mathbb{A}$
quiver $Q_{A}$ with $2n-3$ vertices and an action of a  group $G$ with order $2$ on  
$Q_{A}$ such that the corresponding skew algebra $(\Bbbk Q_{A})G$ is Morita equivalent to $\Bbbk Q_{D}$ (see Sect.~\ref{sec.3} for more details). According to the theory of group action and equivariant, the induced category $(\Bbbk Q_{A}\text{-mod})^{G}$ is equivalent to the left module category  of  $Q_{D}$ (c.f. \cite{C2}). 
Inspired by this equivalence of categories and  the works in \cite{BS2021,2006Quivers,HZZ,S2,CQZ}, we derive a geometric model $P(Q_{D})$ for the left module category  of  $Q_{D}$ via adding a puncture $O$ in the center of the centrally symmetric polygon $P(Q_{A})$, where $P(Q_{A})$ is the polygon model for the module category of  $Q_{A}$ constructed in \cite{BGMR,CQZ}. Specifically,
we define a translation quiver $(\Gamma_{D},R_{D})$ of $P(Q_{D})$, where the vertices  are all the tagged line segments of $P(Q_{D})$; the arrows are induced by pivots of tagged line segments;  $R_{D}$ acts on tagged line segments by clockwise rotation. Let $\mathcal{C}_{D}$ be the mesh category with respect to $(\Gamma_{D},R_{D})$. We get the following result.

\begin{thmIntro}[Theorem \ref{Thm:A}] \label{Thm:A1}
 The category of indecomposable representations over  $Q_{D}$ 
 is equivalent to the category of tagged line segments  $\mathcal{C}_{D}$.
\end{thmIntro}

Employing  this equivalence of categories, we can easily determine the dimension vectors of  indecomposable representations over  $Q_{D}$ with the help  of  tagged line segments.
This approach greatly simplifies  the computation of  the Auslander-Reiten quiver of $\Bbbk Q_{D}$. Furthermore,
 using the geometric model $P(Q_{D})$, we define the
{crossing number} of two tagged line segments on  $P(Q_{D})$ and prove the following theorem.

\begin{thmIntro}[Theorem~\ref{E and e}] \label{Thm:B}
The dimension of the first extension group  between indecomposable representations over  $Q_{D}$  equals   the crossing number of  the corresponding tagged line segments on $P(Q_{D})$.  
\end{thmIntro}

 Besides, the generators of  the first extension group  between  two indecomposable representations over  $Q_{D}$ can be  illustrated via the corresponding tagged line segments on $P(Q_{D})$(see Proposition~\ref{4.28}).

We then introduce the notion of maximal almost pre-rigid representations over  $Q_{D}$. 
In contrast to the  definition of  maximal almost rigid representations, the  prefixion ``pre-" indicates that
the middle term of extension between two indecomposable summands of almost pre-rigid representations   is allowed to  take the form
 $\tau^{-k}\overline{P}_{n-1}\oplus\tau^{-k}\overline{P}_{n}$, for some $k\in \{0,1,\dots,n-2\}$(refer to Definition~\ref{def:arr of D}), while for almost rigid representations, the middle term must be indecomposable.  
 In light of  Proposition~\ref{4.28},
we give a geometric description of them in terms of tagged triangulations of $P(Q_{D})$(see Definition~\ref{tri of P(Q_{D})}).

\begin{thmIntro}[Theorem \ref{Thm 1} ]\label{Thm:C}
There exists  a one-to-one correspondence between the maximal almost pre-rigid representations over  $Q_{D}$ and the tagged triangulations of $P(Q_{D})$. 
\end{thmIntro}

  Consequently, every maximal almost pre-rigid representation   over  $Q_{D}$
has exactly $2n-2$ indecomposable direct summands. 
Based on this, we deduce an application of maximal almost pre-rigid representations in tilting theory. Let $Q_{\overline{D}}$  be the quiver obtained by adding $n-2$ new vertices and $n-2$ arrows to the quiver $Q_{D}$. We  define a functor $G_{D}\colon\Bbbk Q_{ {D}}\text{-mod}\to \Bbbk Q_{\overline{D}}\text{-mod}$ (for a  detailed construction, see Sect. \ref{sec.6}). Our fourth main result is presented as follows.

\begin{thmIntro}[Theorem \ref{Thm:D}]\label{D}
For each  maximal almost pre-rigid representation $\overline{T}$ over  $Q_{D}$,
  its    endomorphism algebra $C$  is a tilted algebra of type $Q_{\overline{D}}$. Moreover, 
there is an isomorphism of algebras 
\[ {\rm End}_{\overline{\mathcal{C}}}  
 \widetilde{G}(\overline{T})\cong C \ltimes{\rm Ext}_{C}^{2}(DC,C),  \]
 where  $C \ltimes{\rm Ext}_{C}^{2}(DC,C)$ is the trivial extension of $C$, $\overline{\mathcal{C}}$  is the cluster category of $Q_{\overline{D}}$ and  $\widetilde{G}_{D}(\overline{T})$ is the image of $G_{D}(\overline{T})$ in $\overline{\mathcal{C}}$ under the natural embedding that maps a representation to its orbit.
\end{thmIntro}

 Furthermore, we explore 
 Cambrian combinatorics on quiver representations over type $\mathbb{D}$ quivers.  
We define a partial order on the set of maximal almost pre-rigid representations over $Q_D$, which  therefore give a representation-theoretic interpretation for the  Cambrian lattice associated with the quiver $Q_D'$, where $Q_D'$ is   obtained from $Q_{D}$ by removing the vertex labeled with $1$ and the arrow $\beta_{1}$. 

\begin{thmIntro}[Theorem \ref{Thm:F}]\label{E}
The type-$\mathbb{D}$  Cambrian lattice coming from $Q_D'$ is  isomorphic to the set of  maximal almost pre-rigid representations over $Q_D$  with the covering relation given in Definition~\ref{def poset}. 
\end{thmIntro}

As a corollary, we   provide a representation-theoretic 
description of the type-$\mathbb{B}$ Cambrian lattice associated with $Q_{D}$.   Precisely,  we define a  covering relation on  the set consisting of    maximal almost pre-rigid representations over $Q_D$ that  contain    $\tau^{-k }\overline{P}_{n-1}\oplus \tau^{-k }\overline{P}_{n}$  as   direct summands for  some   $k\in \{0,1,\dots, n-2\}$,  and then prove that the resulting
  poset  
 is   isomorphic to a 
 type-$\mathbb{B}$ Cambrian lattice.

The paper is organized as follows. After a brief introduction in Sect.~\ref{sec.1}, we recall some needed concepts in Sect.~\ref{sec.2}, and then we investigate three functors: $F_{g}$, $\Phi$, and $\psi$ in Sect.~\ref{sec.3}. Sect.~\ref{sec.4} is devoted to the construction of the  category $\mathcal{C}_{D}$ of tagged line segments on $P(Q_{D})$ and the proof of Theorem~\ref{Thm:A1}.  In Sect.~\ref{sec.5}, we introduce the notion  of  crossing number of tagged line segments on  $P(Q_{D})$ and prove Theorem~\ref{Thm:B}.
 In Sect.~\ref{sec.6}, we define the  maximal almost pre-rigid representations over  $Q_{D}$ and   the tagged  triangulations of $P(Q_{D})$,  and then show the one-to-one relation between them; as stated in Theorem~\ref{Thm:C}.
Furthermore, we construct an additive functor $G_{D}$ and  prove Theorem \ref{D}. 
In Sect.~\ref{sec.7}, a representation-theoretic interpretation of the type-$\mathbb{B}$ and type-$\mathbb{D}$  Cambrian lattices  is given. 
\section{Preliminaries}\label{sec.2}
\subsection{Representations over quivers}
In this subsection, we recall some   notions and a property about representations over a quiver from~\cite{BGMR,ASS}.

 Given a finite connected quiver $Q=(Q_{0}, Q_{1})$, where
$Q_{0}$ is the set of vertices and $Q_{1}$ is the set of arrows.  For two arrows
$\alpha,\beta$ of $Q$   such that the target point of $\alpha$ coincides with the start point of $\beta$, the composition of  $\alpha$  and  $\beta$ is denoted by $\alpha \beta$. 

Let  $ {\rm rep}_{\Bbbk}Q$  denote   the category of finite dimensional representations over $Q$,  and let ${\rm ind} \ Q$ be a full subcategory of $ {\rm rep}_{\Bbbk}Q$ whose objects are one representative of the isoclass of each indecomposable representation. It is well known that the category $\Bbbk Q$-mod of finitely generated left modules over the path algebra $\Bbbk Q$ is equivalent to
the category $ {\rm rep}_{\Bbbk}Q$. Therefore, we always identify these two categories. The Auslander-Reiten quiver $\Gamma_{{\rm rep}_{\Bbbk} Q}$
 of $\Bbbk Q$ has the isoclasses of indecomposable representations as vertices and irreducible morphisms as arrows.
\subsection{Translation quivers and mesh categories}
We review the concepts of  translation quivers and mesh categories, for which there are many references such as  \cite{BGMR,ASS,H88,Rei80}.

A {\em translation quiver} $(\Gamma,\tau)$ is a locally finite quiver $\Gamma=(\Gamma_{0}, \Gamma_{1})$ without loops
 together with an injection $\tau\colon  \Gamma_{0}^{'}\to \Gamma_{0}$ from a subset $ \Gamma_{0}^{'}$ of $\Gamma_{0}$ to $\Gamma_{0}$ satisfying  for all vertices $x\in \Gamma_{0}^{'}$ and $y\in  \Gamma_{0}$,
the number of arrows from $y $  to $x$ is equal to the number of arrows
 from $\tau x$ to $y$. The injection $\tau$ is called the {\em translation}.

Given a  translation quiver $( \Gamma,\tau)$, a \emph{polarization of} $ \Gamma$ is
 an injection $\sigma: \Gamma_{1}^{'}\to \Gamma_{1}$, where $\Gamma_{1}^{'}$ is the set of all arrows $\alpha\colon y\to x$
 with $x \in \Gamma_{0}^{'}$, such that
$\sigma(\alpha)\colon\tau x\to y$  for each arrow $\alpha\colon y\to x\in  \Gamma_{1}^{'}$. Clearly, if $ \Gamma$ contains no multiple arrows, then there is a unique polarization of $ \Gamma$.

Assume that the quiver $ \Gamma$ does not contain any multiple arrows. The {\em path category} of  translation quiver $( \Gamma,\tau)$ is the category whose  objects are the vertices $ \Gamma_0$ of $ \Gamma$, and given $x,y\in \Gamma_{0}$, the $\Bbbk$-vector space of morphisms from $x$ to $y$ is given by the $\Bbbk$-vector space with basis consisting  of all paths from $x $ to $y$. The composition of morphisms is induced from the usual composition of  paths.
The {\em mesh ideal} in the path category of $ \Gamma$ is the ideal generated by the {\em mesh relations}
\begin{equation}\nonumber
m_{x} =\sum_{\alpha:y\to x} \sigma(\alpha) \alpha,\ \forall x \in  \Gamma_{0}^{'}.
\end{equation}
 The {\em mesh category} $\mathcal{M} ( \Gamma,\tau)$ of $( \Gamma,\tau)$ is the quotient of the path category of $( \Gamma,\tau)$ by the mesh ideal. 

\begin{remark}
 Let $Q$ be a quiver of Dynkin type and $\Gamma_{{\rm rep}_{\Bbbk} Q}$ be the Auslander-Reiten quiver of path algebra $\Bbbk Q$. Denote by $\Gamma_{0}^{'}$  the set of all points in $\Gamma_{{\rm rep}_{\Bbbk} Q}$ that correspond to a non-projective indecomposable module, then 
the quiver $\Gamma_{{\rm rep}_{\Bbbk} Q}$ together with Auslander-Reiten translation $\tau:\Gamma_{0}^{'}\to \Gamma_{0}$ is a translation quiver. In this case, the mesh category $\mathcal{M}( \Gamma_{{\rm rep}_{\Bbbk} Q},\tau)$ is equivalent to the category  ${\rm ind}\ Q$.
\end{remark}
\subsection{Sectional paths and maximal slanted rectangles}
We recall the definitions of sectional paths and maximal slanted rectangles as introduced in \cite{S}.
Let $Q$ be a type $\mathbb{A}$ quiver.
A path $M_{0} \to M_{1} \to \dots \to M_{s} $ in the Auslander-Reiten quiver  of  $\Bbbk Q$ is called a \emph{sectional path} if $\tau M_{i+1}\ne M_{i-1}$ for all $i=1,2,\dots,s-1$. Denote by 
$\Sigma_{\rightarrow}(M)$ the set of all indecomposable representations that can be reached from $M$ by a sectional path, see Figure \ref{fig:sectional}.
 And $\Sigma_{\leftarrow}(M)$ be the set of all indecomposable representations from which one can reach $M$ by a sectional path. 

 Denote by $\mathscr{R}_{\rightarrow}(M)$  the set of all indecomposable representations whose position in the Auslander-Reiten quiver is in the slanted rectangle region, where the left boundary is given by $\Sigma_{\rightarrow}(M)$, see Figure~\ref{fig:sectional}. Then $\mathscr{R}_{\rightarrow}(M)$ is called the \emph{maximal slanted rectangle} in the Auslander-Reiten quiver whose leftmost point is $M$. Dual, $\mathscr{R}_{\leftarrow}(M)$ is called the \emph{maximal slanted rectangle} in the Auslander-Reiten quiver whose rightmost point is $M$.
\begin{figure}[h]
\centering
\begin{tikzpicture}[scale=0.7]
\coordinate (A) at (0,0);
\fill[red] (A) circle[radius=1.5pt];
\node (1) at (-.4,0) {\small{$M$}};
\draw (-1,2) -- (3,2);
\draw (-1,-1) -- (3,-1);
\draw[red,thick] (0,0) -- (2,2);
\draw [red,thick](0,0) -- (1,-1);
\node (2) at (1.2,0) {\small{$\Sigma_{\rightarrow}(M)$}};
\end{tikzpicture}
\hspace{1cm}
\begin{tikzpicture}[scale=0.7]
\coordinate (A) at (0,0);
\fill (A) circle[radius=1.5pt];
\node (1) at (-.4,0) {\small{$M$}};
\draw (-1,2) -- (3,2);
\draw (-1,-1) -- (3,-1);
\filldraw [fill=gray!20,draw,thick]
(0,0) -- (2,2)--(3,1)--(1,-1)--(0,0);
\node (2) at (1.5,0.5) {\small{$\mathscr{R}_{\rightarrow}(M)$} };
\end{tikzpicture}
\caption{$\Sigma_{\rightarrow}(M)$(left) and $\mathscr{R}_{\rightarrow}(M)$(right) of $M$}
\label{fig:sectional}
\end{figure}

In \cite[Chapter~3]{S}, the following  properties of maximal slanted rectangles are  provided. For ease of reference, we summarize these properties into a proposition. 

\begin{proposition}\label{2.2}
Let  $M, N$ be two indecomposable representations over $Q$.  
\begin{itemize}
\item[(1)] For each $i\in Q_{0}$, the projective representation $P_{i}$ and injective representation $I_{i}$ associated with $i$ satisfy that $\mathscr{R}_{\rightarrow}(P_{i})=\mathscr{R}_{\leftarrow}(I_{i} )$.
\item[(2)] The dimension of ${\rm Hom}(M,N)$ is either $0$ or $1$, and it  equals  $1$ if and only if $N$ lies in $\mathscr{R}_{\rightarrow}(M)$.
\item[(3)] The dimension of ${\rm Ext}^{1}(M,N)$  is either $0$ or $1$, and it   equals $1$ if and only if $\tau M$  lies in $\mathscr{R}_{\rightarrow}(N)$. 
If the extension group ${\rm Ext}^{1}(M,N)\ne 0$, then a non-zero element of ${\rm Ext}^{1}(M,N)$ can be represented by a non-split short exact sequence of the form $0 \to N \to E \to M \to 0. $ In this case, $\Sigma_{\rightarrow}(N)$ and $\Sigma_{\leftarrow}(M)$ have either $1$ or $2$ points in common, and these points correspond  to the indecomposable summands of $E$.
\end{itemize}
\end{proposition}

\subsection{The categories of equivariant objects and skew group algebras}\label{equivariant category}
In this subsection, 
we summarize  some fundamental knowledge about  the categories of equivariant objects and skew group algebras  from~\cite{R,De,DGNO,CCR}.

Let $\mathscr{A}$ be an additive category and $G$ be a finite group whose unit is denoted by $e$. A $\emph{$G$-action}$ on $\mathscr{A}$ consists of the data $\{ F_{g},\varepsilon_{g,h}|g,h\in G\}$, where each  $F_{g}:\mathscr{A}\to\mathscr{A}$ is an auto-equivalence and each $\varepsilon_{g,h}: F_{g} F_{h}\to F_{gh}$ is a natural isomorphism such that $$\varepsilon_{gh,l}\circ\varepsilon_{g,h}F_{l}=\varepsilon_{g,hl}\circ F_{g}\varepsilon_{h,l}$$ holds for all $g,h,l\in G$. The $G$-action $\{ F_{g},\varepsilon_{g,h}|g,h\in G\}$ is $\emph{strict}$ if each $F_{g}:\mathscr{A}\to\mathscr{A}$ is an automorphism and each isomorphism $\varepsilon_{g,h}$ is the identity transformation.

 Let $\{F_g, \varepsilon_{g,h}|\ g, h\in G\}$ be a $G$-action on $\mathscr{A}$. A \emph{$G$-equivariant object} in $\mathscr{A}$ is a pair $(X, \alpha)$, where $X$ is an object in $\mathscr{A}$ and $\alpha:=\{\alpha_g\}_{g\in G}$  assigns to each $g\in G$ an isomorphism $\alpha_g: X\to F_g(X)$ that satisfies  $ \alpha_{gg'}=(\varepsilon_{g,g'})_X \circ F_g(\alpha_{g'}) \circ \alpha_g.$
A morphism $f\colon(X,\alpha)\rightarrow (Y,\beta)$ is a morphism $f\colon X\rightarrow Y$ in $\mathscr{A}$ such that 
$\beta_g\circ f=F_g(f)\circ \alpha_g$ for all $g\in G$. This gives rise to
the category $\mathscr{A}^G$ of $G$-equivariant objects, and the \emph{forgetful functor} $U\colon \mathscr{A}^G\to \mathscr{A}$ defined by $U(X, \alpha)=X$ and $U(f)=f$. The forgetful functor $U$ admits a left adjoint ${\rm Ind}\colon\mathscr{A} \to \mathscr{A}^G$, which is known as the \emph{induction functor}. 
For an object $X$, set $${\rm Ind}(X)=(\oplus_{h\in G}F_{h}(X),\varepsilon(X)),$$ where for each $g\in G$, the isomorphism
$\varepsilon(X)_{g}:\oplus_{h\in G}F_{h}(X)\to F_{g}(\oplus_{h\in G}F_{h}(X))$ is induced by the isomorphism $(\varepsilon_{g,g^{-1}h})_{X}^{-1}:F_{h}(X)\to F_{g}(F_{g^{-1}h}(X))$ and for a morphism
$\theta:X\to Y$, set $${\rm Ind}(\theta)=\oplus_{h\in G}F_{h}(\theta):{\rm Ind}(X)\to {\rm Ind}(Y).$$

When focus on the category $A$-mod of left $A$-modules, where $A$ is a finite dimensional $\Bbbk$-algebra, let ${\rm Aut}_{k}(A)$ be the group of $\Bbbk$-algebra automorphisms on $A$. We say that $G$ acts on $A$ by $\Bbbk$-algebra automorphisms, if there is a group homomorphism $G\to {\rm Aut}_{\Bbbk}(A)$. In this case, we identify elements in $G$ with their images under this homomorphism. The corresponding \emph{skew group algebra} $AG$ is defined as follows: $AG=\oplus_{g\in G} Au_{g} $ is a free left $A$-module with basis $\{u_g|g \in G\}$ and the multiplication is given by $(au_{g})(bu_{h})=ag(b)u_{gh}$.

For a $\Bbbk$-algebra automorphism $g$ on $A$ and an $A$-module $M$, the {\rm twisted module} $^{g}M$ is defined such that $^{g}M=M$ as a vector space and that the new $A$-action $``\circ"$ is given by $a\circ m= g(a)m$. This gives rise to a $\Bbbk$-linear automorphism 
$^{g}(-): A \text{-mod}\to A \text{-mod}, $
 which  acts on morphism by the identity.
It is easy to see that for two $\Bbbk$-algebra automorphisms $g$ and $h$ on $A$, $^{h}(^{g}M)=\ ^{gh}M$ holds for any $A$-module $M$. Then there is a strict $\Bbbk$-linear $G$-action on $A\text{-mod}$ by setting $F_{g}=\ ^{g^{-1}}(-)$. Moreover, there is an equivalence of categories
$$\Phi:(A\text{-mod})^{G}\to AG\text{-mod},$$
by sending a $G$-equivariant object $(X,\alpha)$ to the $AG$-module $X$, where the $AG$-module structure is given by $(au_g)x=a\alpha_{g^{-1}}(x)$. Using this equivalence, the induction functor ${\rm Ind}\colon\mathscr{A} \to \mathscr{A}^G$ is identified with $AG\otimes_{A}{-}:A\text{-mod}\rightarrow AG\text{-mod}$.
\section{Some important functors} \label{sec.3}
Let $Q_{A_{2n-3}}$ denote the quiver with the underlying graph depicted in
Figure \ref{fig:A2n-3} and directional symmetry. Specifically, for all $\alpha_i\in Q_{A_{2n-3}, \ 1}$,
  the  arrows 
 $\alpha_i$ and $\alpha_{2n-3-i}$ share the same direction. 
 \begin{figure}[h] 
 \centering
\begin{tikzpicture}[scale=0.9]
\node (0) at (-1.65,0) {\small$\overline{Q}_{A_{2n-3}}=$};
\node (1) at (-0.2,0) {\small{$n-1$}};
\node (2) at (0.9,1) {\small{$n-2$}};
\node (b) at (1,-1) {\small{$n$}};
\node (3) at (2.7,1) {\small{$n-3$}};
\node (c) at (2.7,-1) {\small{$n+1$}};
\node (4) at (4.1,1) {\dots};
\node (d) at (4.1,-1) {\dots};
\node (5) at (5.5,1) {\small{$2$}};
\node (e) at (5.5,-1) {\small{$2n-4$}};
\node (6) at (7.5,1) {\small{$1$}};
\node (f) at (7.5,-1) {\small{$2n-3$}}; 
\draw (1)--(2)-- (3)-- (4) -- (5)-- (6) ;
\draw (1) -- (b) -- (c)-- (d) -- (e)-- (f) ;
\draw (-0.6,0.9) node [anchor=north west] [align=left] {\small$\alpha_{n-2}$};
\draw (-0.6,-0.4) node [anchor=north west] [align=left] {\small$\alpha_{n-1}$};
\draw (1.4,1.55) node [anchor=north west] [align=left] {\small$\alpha_{n-3}$};
\draw (1.4,-1.0) node [anchor=north west] [align=left] {\small$\alpha_{n}$};
\draw (6,1.45) node [anchor=north west] [align=left] {\small$\alpha_{1}$};
\draw (6,-1.00) node [anchor=north west] [align=left] {\small$\alpha_{2n-4}$};
\end{tikzpicture}
    \caption{The underlying graph of $Q_{A_{2n-3}}$}
    \label{fig:A2n-3}
\end{figure}
Let $G$ be a cyclic group of order 2 with generator $g$, where $g$ is the automorphism of  $Q_{A_{2n-3}}$ given by $g(i) =2n-2-i$  and
$g(\alpha_{i}) =\alpha_{2n-3-i}$ for each $i\in Q_{A_{2n-3}, \ 0}$ and $\alpha_i\in Q_{A_{2n-3}, \ 1}$.  
By~\cite[Section~2]{R}, the skew group algebra $(\Bbbk Q_{A_{2n-3}})G$ is Morita equivalent to the path algebra $\Bbbk Q_{D_{n}}$, where the underlying graph  of the quiver $Q_{D_{n}}$ is  shown in Figure~\ref{fig:Dn}
and the direction of the arrow $\beta_{i}$ is the same as $\alpha_{i}$ for all $i\in\{1,2,\dots,n-1\}$. Therefore,  $Q_{D_{n}}$ is a type $\mathbb{D}$ quiver with directional symmetry.
\begin{figure}[h]  \centering
\begin{tikzpicture}
\node (0) at (-1,1) {\small$\overline{Q}_{D_{n}}=$};
\node (1) at (0,2) {\small$n-1$};
\node (a) at (0,0) {\small$n$};
\node (2) at (1,1) {\small$n-2$};
\node (3) at (3.5,1) {\small$n-3$};
\node (4) at (5,1) {\dots};
\node (5) at (6,1) {\small$2$};
\node (6) at (7.5,1) {\small$1$}; 
\draw (1)--(2)-- (3)-- (4) -- (5)-- (6) ;
\draw (a) -- (2)  ;
\draw (0.5,1.9) node [anchor=north west] [align=left] {\small$\beta_{n-2}$};
\draw (0.5,0.6) node [anchor=north west] [align=left] {\small$\beta_{n-1}$};
\draw (1.8,1.5) node [anchor=north west] [align=left] {\small$\beta_{n-3}$};
\draw (6.5,1.5) node [anchor=north west] [align=left] {\small$\beta_{1}$};
\end{tikzpicture}
\caption{The underlying graph of $Q_{D_{n}}$}
    \label{fig:Dn}
\end{figure}

For simplification of the notations,  we abbreviate $Q_{A_{2n-3}}$ as $Q_{A}$ and $Q_{D_{n}}$ as $Q_{D}$, denote the category $\Bbbk Q_{A}\text{-mod}$ by $\mathcal{A}$ and the category $\Bbbk Q_{D}\text{-mod}$ by  $\mathcal{D}$  in the rest of the paper. 

According to the theory of group actions, the action of $G$ on $Q_{A}$ induces a strict action $\{ F_{g},\varepsilon_{g,h}|g,h\in G\}$ of $G$ on $\mathcal{A}$, where the automorphism $F_{g}: \mathcal{A} \to \mathcal{A} $ 
is determined thus: for any projective module $P_{i}$ corresponding to $i\in Q_{A,\ 0}$ and $k\in\{0,1,\dots,n-2\}$
\begin{equation}\label{F_{g}}
F_{g}(\tau^{-k} P_{i})=\tau^{-k} P_{2n-2-i}.
\end{equation}

Next, we will discuss  the functors discussed in Section \ref{equivariant category} that associated with the strict  action  $\{ F_{g},\varepsilon_{g,h}|g,h\in G\}$  on the category $\mathcal{A}$. To  characterize these functors better, we initially provide the following two observations.

\begin{remark}\label{3.2}
(1) Denote by  $e_{j}$ the primitive orthogonal idempotent of $\Bbbk Q_{A}$ corresponding to $j\in Q_{A,0}$.
By the definition of $F_{g}$, we get that equation (\ref{Dim A}) holds for all $i,j\in Q_{A,0}$.
\begin{equation}\label{Dim A}
{\rm dim}_{\Bbbk} e_{j} (\tau^{-k}P_{i})={\rm dim}_{\Bbbk} e_{2n-2-j} ( \tau^{-k}P_{2n-2-i} ). 
\end{equation}

(2) Denote by $\overline{e}_{j}$ the primitive orthogonal idempotent of $\Bbbk Q_{D}$ corresponding to $j\in Q_{D,0}$ and
by $ \overline{P}_{i}$  the projective module corresponding to $i\in Q_{D,\ 0}$. Then  
\[{\rm dim}_{\Bbbk}\overline{e}_{j}(\tau^{-k}\overline{P}_{n-1})=\begin{cases}
{\rm dim}_{\Bbbk}e_{j}(\tau^{-k}P_{n-1}),  &{\rm for}\ j=1,2,\dots ,n-2,\\
1-(k\ {\rm mod} \ 2),  &{\rm for}\ j=n-1,\\
k\ {\rm mod} \ 2,&{\rm for}\ j=n,\\
\end{cases}\]
\[{\rm dim}_{\Bbbk}\overline{e}_{j}(\tau^{-k}\overline{P}_{n})=
\begin{cases}
{\rm dim}_{\Bbbk}e_{j}(\tau^{-k}P_{n-1}),  &{\rm for}\ j=1,2,\dots ,n-2,\\
k\ {\rm mod} \ 2,  &{\rm for}\ j=n-1,\\
1-(k\ {\rm mod} \ 2),&{\rm for}\ j=n,\\
\end{cases}\]
and for  $i\in Q_{D,\ 0}\setminus\{n-1,n\}$,  
\[{\rm dim}_{\Bbbk}\overline{e}_{j}(\tau^{-k}\overline{P}_{i})=\begin{cases}
{\rm dim}_{\Bbbk}e_{j}(\tau^{-k}P_{i})+{\rm dim}_{\Bbbk}e_{2n-2-j}(\tau^{-k}P_{i}),     &{\rm for}\ j=1,2,\dots ,n-2, \\
{\rm dim}_{\Bbbk}e_{n-1}(\tau^{-k}P_{i}),   &{\rm for}\ j=n-1,\\
{\rm dim}_{\Bbbk}e_{n-1}(\tau^{-k}P_{i}),   &{\rm for}\ j=n.\\
\end{cases}\]
\end{remark}
Based on the above remark, by performing calculations, we can obtain the action of the following functors on objects.

\begin{itemize}
\item[(1)] 
The equivalent functor  
$\Phi:\mathcal{A}^{G}\to \mathcal{D}$ satisfies
\[
\Phi \left(P_{n-1},\alpha\right)=\overline{P}_{n-1}
,\   \Phi \left(P_{n-1},\beta\right)=\overline{P}_{n},\ \Phi \left(P_{i}\oplus P_{2n-2-i},\delta\right)=\overline{P}_{i},\ \forall i\in \{1,2,\dots,n-2\},
\]
where 
\[\begin{cases}
\alpha_{e}=1,\\
\alpha_{g}=1,\\
\end{cases}
\ 
\begin{cases}
\beta_{e}=1,\\
\beta_{g}=-1,\\
\end{cases}
\ \textup{and} \quad 
\begin{cases}
\delta_{e}=\left(\begin{smallmatrix}1& 0\\0&1 \end{smallmatrix}\right),\\
\delta_{g}=\left(\begin{smallmatrix}0& 1\\1&0 \end{smallmatrix}\right).\\
\end{cases}
\]
\item[(2)] 
The functor   $ \psi:=(\Bbbk Q_{A})G\otimes_{\Bbbk Q_{A}}{-}:\mathcal{A}\to \mathcal{D}$ is  exact by~\cite[Section 2.1]{R} and we have
\[
\psi (\tau^{-k}P_{n-1})=\tau^{-k}\overline{P}_{n-1}\oplus \tau^{-k}\overline{P}_{n},\ \psi (\tau^{-k}P_{i})=\tau^{-k}\overline{P}_{i}=\psi (\tau^{-k}P_{2n-2-i}),\ \forall i\in \{1,2,\dots,n-2\}.
\]
\end{itemize}

 Toward the end of this section, we investigate the relationship between the first extension groups  under the action of  $F_g$, yielding the following properties.

\begin{proposition}\label{prop:E}
Let  $M$ and $N$ be two indecomposable representations over $Q_{A}$.  
 \begin{enumerate}
\item [(1)] ${\rm Ext}_{\mathcal{A}}^{1}(M,F_{g}(M))=0={\rm Ext}_{\mathcal{A}}^{1}(F_{g}(M),M)$;
\item [(2)] The functor $F_g$ induces a bijection between ${\rm Ext}_{\mathcal{A}}^{1}(M, N)$ and ${\rm Ext}_{\mathcal{A}}^{1}(F_{g}(M),F_{g}(N))$. In particular, ${\rm Ext}_{\mathcal{A}}^{1}(M,N)=0$ if and only if 
${\rm Ext}_{\mathcal{A}}^{1}(F_{g}(M),F_{g}(N))=0$.
\end{enumerate}
\end{proposition}
\begin{proof}
(1)  For each indecomposable  representation   $M$ over $Q_{A}$, $\tau M\notin\mathcal{R}_{\rightarrow}(F_{g}(M))$ and $\tau F_{g}(M)\notin\mathcal{R}_{\rightarrow}(M)$. Thus this statement is a direct consequence of Proposition \ref{2.2}(3). 

(2) Note   that for any two indecomposable  representations  $M, N$ over $Q_{A}$,  the dimension of ${\rm Ext}_{\mathcal{A}}^{1}(M, N)$ is either $0$ or $1$. Consequently,  the dimension of ${\rm Ext}_{\mathcal{A}}^{1}(F_{g}(M),F_{g}(N))$ is either $0$ or $1$, since both $F_{g}(M)$ and $F_{g}(N)$ are also indecomposable  representations over $Q_{A}$. This proof will be completed by showing  that ${\rm dim}_{\Bbbk}{\rm Ext}_{\mathcal{A}}^{1}(M,N)=1$ if and only if  ${\rm dim}_{\Bbbk}{\rm Ext}_{\mathcal{A}}^{1}(F_{g}(M),F_{g}(N))=1.$

If ${\rm dim}_{\Bbbk}{\rm Ext}_{\mathcal{A}}^{1}(M,N)=1$,   let the  short exact sequence $\xi$ be a generator of ${\rm Ext}_{\mathcal{A}}^{1}(M, N)$. 
Since the functor $F_g$ is   exact, $F_g(\xi)$ is  a non-split short exact sequence in ${\rm Ext}_{\mathcal{A}}^{1}(F_{g}(M),F_{g}(N))$. 
 On the other hand, if 
 ${\rm dim}_{\Bbbk}{\rm Ext}_{\mathcal{A}}^{1}(F_{g}(M),F_{g}(N))=1,$ then the generator $\xi^{'}$ of ${\rm Ext}_{\mathcal{A}}^{1}(F_{g}(M),F_{g}(N))$ has the form 
$0 \to F_{g}(N) \to F \to F_{g}(M) \to 0$. Since   $F_g$ satisfies $F_{g}\circ F_{g}={\rm Id}_{\mathscr{A}}$,  it follows that $F_g(\xi^{'})$
 is a non-split short exact sequence in ${\rm Ext}_{\mathcal{A}}^{1}(M,N)$. 
\end{proof}


\section{Geometric models for $\Bbbk Q_{A}\text{-}{\rm mod}$ and $\Bbbk Q_{D}\text{-}{\rm mod}$}\label{sec.4}

In this section, we recall a geometric construction for the category  $\mathcal{A}$ in terms of a convex polygon  from 
 \cite{BGMR,CQZ}. And then we combine this polygon  with the relationship between $Q_{A}$ and $Q_{D}$ provided in Section \ref{sec.3}, to give a geometric model for the category $\mathcal{D}$.   

\subsection{A geometric model for $\Bbbk Q_{A}\text{-mod}$}\label{subsec.4.1}
In \cite{CQZ}, a $(2n-2)$-gon $P(Q_{A})$ was constructed  via $Q_{A}$ as follows: First, define the vertices by setting 
\[X_{s}=(x_{s},y_{s})\in \{(x,y)|x^{2}+y^{2}=1\} \] for each $s\in\{0,1,\dots,2n-3\}$ with 
 \[
y_{s}=-1+\frac{2s}{2n-3}\ \ {\rm and} 
  \ \
x_{s}={\rm sign}(X_{s})\sqrt{1-y_{s}^{2}},
\] 
 where $x_{0}=0=x_{2n-3}$ and for $s\in\{1,2,\dots,2n-4\}$, 
\[{\rm sign}(X_{s})=
\begin{cases}
+,    & {\rm if} \ s\xrightarrow{\alpha_{s}} s+1 \in Q_{A, \ 1};\\
-,   &{\rm if} \ s\xleftarrow{\alpha_{s}}  s+1\in Q_{A, \ 1}.\\
\end{cases}
\]  
Next, starting  from vertex $X_{0}$, connect   the vertices in the set $\{X_{s}|s=0,1,\dots, 2n-3
\}$  sequentially in a clockwise direction to form  a closed  figure, denoted as $P(Q_{A})$.
It is evident that $P(Q_{A})$ is a centrally symmetric convex $(2n-2)$-gon.
Precisely, the vertices $X_{s}$ and  $X_{2n-3-s}$ are always  symmetric about the center $O=(0,0)$ of $P(Q_{A})$ for all $s\in\{0,1, \dots,2n-3\}$. See Figure \ref{fig:A} for an example.
\begin{figure}[h]
\centering
\begin{tikzpicture}[scale=0.9]
\node (0) at (-1,0) {\small$Q_{A_{7}}=$};
\node (1) at (0,0) {\small$4$};
\node (2) at (1,1) {\small$3$};
\node (b) at (1,-1) {\small$5$};
\node (3) at (2,1) {\small$2$};
\node (c) at (2,-1) {\small$6$};
\node (3) at (3,1) {\small$1$};
\node (c) at (3,-1) {\small$7$};
\draw (1.1,1.43) node [anchor=north west] [align=left] {\small$\alpha_{2}$};
\draw (2.2,1.43) node [anchor=north west] [align=left] {\small$\alpha_{1}$};
\draw (-0.15,0.9) node [anchor=north west] [align=left] {\small$\alpha_{3}$};
\draw (-0.1,-0.4) node [anchor=north west] [align=left] {\small$\alpha_{4}$};
\draw (1.2,-1) node [anchor=north west] [align=left] {\small$\alpha_{5}$};
\draw (2.2,-1) node [anchor=north west] [align=left] {\small$\alpha_{6}$};
\draw[<-] (0.2,0.15) -- (0.75,0.95);
\draw[<-] (0.2,-0.15) -- (0.75,-0.95);
\draw[->] (1.25,1) -- (1.75,1);
\draw[->] (1.25,-1) -- (1.75,-1);
\draw[<-] (2.25,1) -- (2.75,1);
\draw[<-] (2.25,-1) -- (2.75,-1);
\end{tikzpicture}
\hspace{2cm}
\begin{tikzpicture}[scale=0.9 ]
\node (0) at (-2.5,0) {\small $P(Q_{A_{7}})=$};
\node (1) at (0,.9){};
\node (2) at (-0.67,0.72){};
\node (3) at (0.9,0.44){};
\node (4) at (-0.987,-0.16){};
\node (5) at (0.987,-0.16){};
\node (6) at (-0.9,-0.44){};
\node (7) at (0.67,-0.72){};
\node (8) at (0,-.95){};
		\draw (1) node[above] {\small$X_{7}$};
		\draw (2) node[left]{\small$X_{6}$};
		\draw (3) node[right] {\small$X_{5}$};	
\node (9) at (-1.25,-0.0){\small$X_{4}$};
		\draw (5) node[right] {\small$X_{3}$};
		\draw (6) node[below left] {\small$X_{2}$};
           \draw (7) node[right] {\small$X_{1}$};
		\draw (8) node[below] {\small$X_{0}$};

\draw (0,1) -- (0.9,0.44)--(0.987,-0.16)--(0.67,-0.72)--(0,-1)--(-0.9,-0.44)--(-0.987,0.16)--(-0.67,0.72)--(0,1);
\end{tikzpicture}
\caption{The quiver $Q_{A_{7}}$ and the polygon $P(Q_{A_{7}})$}
\label{fig:A}
\end{figure}

Now we   recall  from \cite{BGMR}  the geometric realization for the category $\mathcal{A}$ in terms of $(2n-2)$-gon $P(Q_{A})$. 
First, let us establish some notation conventions: for any vertex $X_{s}$ of $P(Q_{A})$, we denote by $X_{R_{A}(s)}$  the clockwise of $X_{s}$ on the boundary of $P(Q_{A})$, and by $X_{R_{A}^{-1}(s)}$ 
 counterclockwise neighbor of $X_{s}$.  Define the set $\omega$ as follows
\[
\omega =\{\gamma(s,t)\mid -1\le y_{s}<y_{t}\le 1\},\]
where  $\gamma(s,t)$ is the line segment of $P(Q_{A})$ connecting the vertices $X_{s}$ and $X_{t}$.

 For each $\gamma(s,t)\in \omega$, the line segment 
$\gamma(s,R_{A}^{-1}(t))$ is called a \emph{pivot} of 
$\gamma(s,t)$ if it lies in the set $\omega$. Similarly, the line segment  $\gamma(R_{A}^{-1}(s),t)$ is a  \emph{pivot} of $\gamma(s,t)$ if it lies in $\omega$. 
\begin{remark}\label{rek:A-iif} 
Keep the notation as above.
For each $\gamma(s,t)\in \omega$, $\gamma(u,v)$ is a pivot of $\gamma(s,t)$ if and only if $\gamma({2n-3-v},{2n-3-u})$ is a pivot of $\gamma({2n-3-t},{2n-3-s})$. 
\end{remark}

Define a translation quiver $(\Gamma_{A},R_{A})$ of $P(Q_{A})$ with respect to pivots, where the vertices of $\Gamma_{A}$ are all the line segments in
 $\omega$; there is an arrow $\gamma(s,t)\to\gamma({u},{v})$ in $\Gamma_{A}$ if and only if $\gamma({u},{v})$ is a pivot of $\gamma(s,t)$, and
 $R_{A}$ is a translation on the set $\omega$ with
\[ 
R_{A}(\gamma(s,t))=
\begin{cases}
\gamma(R_{A}(s),R_{A}(t))     &{\rm if} \  \gamma(R_{A}(s),R_{A}(t))\in \omega; \\
0   & \text{otherwise}.\\
\end{cases}\]

Let $\mathcal{C}_{A}$ be the mesh category of the translation quiver $(\Gamma_{A},R_{A})$.  
Then there is a functor $F_{A}$  from $\mathcal{C}_{A}$ to the abelian category ${\rm ind}\ Q_{A}$ defined as follows: 
On objects,
$F_{A}(\gamma(s,t))=M(s+1,t),$
 where $M(s+1,t)$ is the indecomposable representation supported on the vertices between $s+1$ and $t$.
 On  morphisms,  define $F_{A}\left(\gamma(s,t)\to \gamma(R_{A}^{-1}(s),t)\right) $ to  be the irreducible morphism $M(s+1,t)\to M(R_{A}^{-1}(s)+1,t)$. 
 Similarly,  let $F_{A}\left(\gamma(s,t)\to \gamma(s,R_{A}^{-1}(t))\right) $   be the irreducible morphism  $M(s+1,t)\to M(s+1,R_{A}^{-1}(t))$, see \cite[Definition~4.5]{BGMR} for a detailed definition.

\begin{theorem}[{See  \cite{BGMR}}]\label{thm:FA} 
The functor $ F_{A}:\mathcal{C}_{A}\to {\rm ind}\ Q_{A}$ is an equivalence of categories. 
\end{theorem}

\subsection{A geometric model for $\Bbbk Q_{D}\text{-mod}$}
In this subsection, we will derive a geometric model for the category $\mathcal{D}$ via $P(Q_{A})$ and  equivariant theory of group actions.

\subsubsection{The punctured polygon $P(Q_{D})$}
Recall that  there is a unique $Q_{A}$ associated with the given $Q_{D}$ and
$P(Q_{A})$ is the geometric model for the category $\mathcal{A}$.  We relabel the vertices $\{X_{s}|s=0,1,\dots, 2n-3\}$ of  $P(Q_{A})$  according to the following rule $(*)$:
\begin{itemize}
\item[-] if $y_s<0$, then relabel the vertex $X_{s}$ by $Y_{s-n+1}$;
\item[-] if $y_s>0$, then relabel the vertex $X_{s}$ by $Y_{s-n+2}$.
\end{itemize}
 Denote the punctured  polygon with vertices $\{Y_{\pm s}|s=1,2,\dots ,n-1\}$ and a puncture 
 $O=(0,0)$ by $P(Q_{D})$. Then the vertices $Y_{s}$ and  $Y_{-s}$ are symmetric about $O$ for all $s\in\{1,2,\dots ,n-1\}$.

 To illustrate the construction of $P(Q_{D})$, we present an example here.
 \begin{example}
Let $Q_{D_{5}}$ be the quiver on the left side of Figure~\ref{fig:D}. Then 
the associated quiver $Q_{A_{7}}$ and   polygon $P(Q_{A_{7}})$ are shown in Figure~\ref{fig:A}. After relabeling the vertices of $P(Q_{A_{7}})$ and adding a puncture at its center, we obtain the punctured polygon $P(Q_{D_{5}})$, illustrated on the right side of Figure~\ref{fig:D}.
\begin{figure}[h]
\centering
\begin{tikzpicture}[scale=0.96]
\node (0) at (-1,0) {\small $Q_{D_{5}}=$};
\node (1) at (0,1) {\small$4$};
\node (a) at (0,-1) {\small$5$};
\node (2) at (1,0) {\small$3$};
\node (3) at (2,0) {\small$2$};
\node (4) at (3,0) {\small$1$};
\draw (1.1,0.53) node [anchor=north west] [align=left] {\small$\beta_{2}$};
\draw (2.2,0.53) node [anchor=north west] [align=left] {\small$\beta_{1}$};
\draw (0.5,1) node [anchor=north west] [align=left] {\small$\beta_{3}$};
\draw (0.5,-0.3) node [anchor=north west] [align=left] {\small$\beta_{4}$};
\draw[<-] (0.2,0.85) -- (0.75,0.15);
\draw[<-] (0.2,-0.85) -- (0.75,-0.15);
\draw[->] (1.25,0) -- (1.75,0);
\draw[<-] (2.25,0) -- (2.75,0);
\end{tikzpicture}
\hspace{1cm}
\begin{tikzpicture}[scale=0.9]
\node (0) at (-2.5,-0.25) {\small $P(Q_{D_{5}})=$};
\node (1) at (0,0.9){};
\node (2) at (-0.67,0.72){};
\node (3) at (0.9,0.44){};
\node (4) at (-0.987,-0.16){};
\node (5) at (0.987,-0.16){};
\node (6) at (-0.9,-0.44){};
\node (7) at (0.67,-0.72){};
\node (8) at (0,-0.95){};
\node (9) at (0,0){};
\fill (9) circle(0.05);
\draw (1) node[above] {\small$Y_{4}$};
		\draw (2) node[left]{\small$Y_{3}$};
		\draw (3) node[right] {\small$Y_{2}$};	
\node (10) at (-1.25,-0.0){\small$Y_{1}$};
		\draw (5) node[right] {\small$Y_{-1}$};
		\draw (6) node[below left] {\small$Y_{-2}$};
           \draw (7) node[right] {\small$Y_{-3}$};
		\draw (8) node[below] {\small$Y_{-4}$};
\draw (9) node[above left] {\small$O$};
\draw (0,1) -- (0.9,0.44)--(0.987,-0.16)--(0.67,-0.72)--(0,-1)--(-0.9,-0.44)--(-0.987,0.16)--(-0.67,0.72)--(0,1);
\end{tikzpicture}
\caption{The quiver $Q_{D_{5}}$ and the polygon $P(Q_{D_{5}})$}
\label{fig:D}
\end{figure}
\end{example}

\subsubsection{The category of tagged line segments $\mathcal{C}_{D}$} Now we will define a category  $\mathcal{C}_{D}$ with respect to the line segments of $P(Q_D)$ and prove the equivalence between   $\mathcal{C}_{D}$ and ${\rm ind}\ Q_{D}$.
We first 
construct a set $\Omega$ as follows:
\[
\Omega=\{\gamma_{s}^{t}\mid -(n-1)\leq s<t\leq n-1, s\ne -t\}/_{\sim} \cup \{\gamma_{-t}^{t,-1}, \gamma_{-t}^{t,1}\mid 1\leq t\leq n-1\},
\]
where $\gamma_{s}^{t}$ is the unique line segment between $Y_{s}$ and $Y_{t}$ of $P(Q_D)$  with $\gamma_{s}^{t}\sim\gamma_{-t}^{-s}$, and 
 $\gamma_{-t}^{t,\epsilon}\  (\epsilon=1, -1)$ 
denotes the  two overlapping line segments between  $Y_{-t}$ and $Y_{t}$. For distinguishing the two segments, we always draw the line segment $\gamma_{-t}^{t,-1}$ in picture  with a tag  $``|"$ on it and the line segment $\gamma_{-t}^{t,1}$ in picture with no tag.

For example, consider the punctured polygon $P(Q_{D_{5}})$. In this case, there exists $\gamma_{1}^{4}\sim\gamma_{-4}^{-1}$ in $\Omega$.   Figure~\ref{fig:-1,1} illustrates $\gamma_{-3}^{3, -1}$ and $\gamma_{-3}^{3, 1}$ of  $P(Q_{D_{5}})$  in red. Specifically, $\gamma_{-3}^{3,-1}$ is tagged with  $``|"$, while $\gamma_{-3}^{3,1}$ remains untagged.

\begin{figure}[h]
\centering
\begin{tikzpicture}[scale=0.9]
\node (1) at (0+12.5,1-6){};
\node (2) at (-0.67+12.5,0.72-6){};
\node (3) at (0.9+12.5,0.44-6){};
\node (4) at (-0.987+12.5,0.16-6){};
\node (5) at (0.987+12.5,-0.16-6){};
\node (6) at (-0.9+12.5,-0.44-6){};
\node (7) at (0.67+12.5,-0.72-6){};
\node (8) at (0+12.5,-1-6){};
\node (9) at (0+12.5,0-6){};
\fill (9) circle(0.07);
\node (10) at (0+12.2,0-6.3){{\fontsize{0.27em}{0.32em}\selectfont$\gamma_{-3}^{3,-1}$}};
\draw (9) node[above right][inner sep=0.75pt] {{\fontsize{0.27em}{0.32em}\selectfont$O$}};
		\draw (1) node[above][inner sep=0.75pt] {{\fontsize{0.27em}{0.32em}\selectfont$Y_{4}$}};
		\draw (2) node[left][inner sep=0.75pt]{{\fontsize{0.27em}{0.32em}\selectfont$Y_{3}$}};
		\draw (3) node[right][inner sep=0.75pt] {{\fontsize{0.27em}{0.32em}\selectfont$Y_{2}$}};
\node (9) at (-1.25+12.5,-0.0-6){{\fontsize{0.27em}{0.32em}\selectfont$Y_{1}$}};
		\draw (5) node[right][inner sep=0.75pt] {{\fontsize{0.27em}{0.32em}\selectfont$Y_{-1}$}};
		\draw (6) node[below left][inner sep=0.75pt] {{\fontsize{0.27em}{0.32em}\selectfont$Y_{-2}$}};
           \draw (7) node[right][inner sep=0.75pt] {{\fontsize{0.27em}{0.32em}\selectfont$Y_{-3}$}};
		\draw (8) node[below][inner sep=0.75pt] {{\fontsize{0.27em}{0.32em}\selectfont$Y_{-4}$}};
\draw[line width=1.2pt,red] (-0.67+12.65,0.72-6.35)--(-0.67+12.85,0.72-6.15);
\draw (0+12.5,1-6) -- (0.9+12.5,0.44-6)--(0.987+12.5,-0.16-6)--(0.67+12.5,-0.72-6)--(0+12.5,-1-6)--(-0.9+12.5,-0.44-6)--(-0.987+12.5,0.16-6)--(-0.67+12.5,0.72-6)--(0+12.5,1-6);
\draw[->,line width=1.2pt,red] (0.67+12.5,-0.72-6)--(-0.67+12.5,0.72-6);
\draw[line width=1.2pt,red] (0.67+12.5,-0.72-6)--(0+12.5,0-6);
\fill (2) circle(0.05);
\fill (7) circle(0.05);
\end{tikzpicture}
\hspace{1cm}
\begin{tikzpicture}[scale=0.9]
\node (1) at (0+12.5,1-6){};
\node (2) at (-0.67+12.5,0.72-6){};
\node (3) at (0.9+12.5,0.44-6){};
\node (4) at (-0.987+12.5,0.16-6){};
\node (5) at (0.987+12.5,-0.16-6){};
\node (6) at (-0.9+12.5,-0.44-6){};
\node (7) at (0.67+12.5,-0.72-6){};
\node (8) at (0+12.5,-1-6){};
\node (9) at (0+12.5,0-6){};
\node (10) at (0+12.3,0-6.3){{\fontsize{0.27em}{0.32em}\selectfont$\gamma_{-3}^{3,1}$}};
\fill (9) circle(0.07);
\draw (9) node[above right][inner sep=0.75pt] {{\fontsize{0.27em}{0.32em}\selectfont$O$}};
		\draw (1) node[above][inner sep=0.75pt] {{\fontsize{0.27em}{0.32em}\selectfont$Y_{4}$}};
		\draw (2) node[left][inner sep=0.75pt]{{\fontsize{0.27em}{0.32em}\selectfont$Y_{3}$}};
		\draw (3) node[right][inner sep=0.75pt] {{\fontsize{0.27em}{0.32em}\selectfont$Y_{2}$}};
\node (9) at (-1.25+12.5,-0.0-6){{\fontsize{0.27em}{0.32em}\selectfont$Y_{1}$}};
		\draw (5) node[right][inner sep=0.75pt] {{\fontsize{0.27em}{0.32em}\selectfont$Y_{-1}$}};
		\draw (6) node[below left][inner sep=0.75pt] {{\fontsize{0.27em}{0.32em}\selectfont$Y_{-2}$}};
           \draw (7) node[right][inner sep=0.75pt] {{\fontsize{0.27em}{0.32em}\selectfont$Y_{-3}$}};
		\draw (8) node[below][inner sep=0.75pt] {{\fontsize{0.27em}{0.32em}\selectfont$Y_{-4}$}};

\draw (0+12.5,1-6) -- (0.9+12.5,0.44-6)--(0.987+12.5,-0.16-6)--(0.67+12.5,-0.72-6)--(0+12.5,-1-6)--(-0.9+12.5,-0.44-6)--(-0.987+12.5,0.16-6)--(-0.67+12.5,0.72-6)--(0+12.5,1-6);
\draw[->,line width=1.2pt,red] (0.67+12.5,-0.72-6)--(-0.67+12.5,0.72-6);
\draw[line width=1.2pt,red] (0.67+12.5,-0.72-6)--(0+12.5,0-6);
\fill (2) circle(0.05);
\fill (7) circle(0.05);
\end{tikzpicture}
\caption{The line segment $\gamma_{-3}^{3,-1}$(left) and $\gamma_{-3}^{3,1}$(right) of $P(Q_{D_{5}})$}
\label{fig:-1,1}
\end{figure}

\begin{remark}
 According to the  definition of $\Omega$, it is easy to calculate that there are $n(n-1)$  
elements in $\Omega$. To facilitate the presentation of other definitions, we denote the elements in $\Omega$ by $\gamma_{s}^{t,l}$ and call them {\emph{tagged line segments}}. Precisely, $\gamma_{s}^{t,l}=\gamma_{s}^{t}$ for the cases $s\neq -t$, and $\gamma_{s}^{t,l}=\gamma_{-t}^{t,\epsilon}$ for the cases $s= -t$.  
\end{remark}

 Let $ [s,t]$ denote  the number of the vertices on the minimal path from $Y_{s}$ to $Y_{t}$ along the boundary of $P(Q_{D})$  in counterclockwise direction (including $Y_{s}$ and $Y_{t}$). For example, see the right side of Figure~\ref{fig:D}, the minimal path from $Y_{-2}$ to $Y_{4}$ along the boundary of $P(Q_{D_5})$ in counterclockwise direction is $Y_{-2}\to Y_{-4}\to Y_{-3}\to Y_{-1}\to Y_{2}\to Y_{4}$, which consists of 6 vertices, so $[-2,4]=6$ on $P(Q_{D_{5}})$. 

Notice that the equation $[s,t]+[-t,-s]=2n$ holds for all $-(n-1)\leq s<t\leq n-1$. Consequently,  either $[s,t]$ or $[-t,-s]$ must be greater than or equal to $n$, while the other is less than or equal to $n$. Additionally, the tagged line segments $\gamma_{s}^{t}$ and $\gamma_{-t}^{-s}$ satisfy the relation $\gamma_{s}^{t}=\gamma_{-t}^{-s}$ in $\Omega$. Hence,  
it suffices to consider only one representative from each pair, when we discuss  the  tagged line segments in $\Omega$. 
For the rest of this section, we always select the representative elements such that $[- ,- ]\leq n$, unless otherwise specified.

 \begin{definition}\label{pivot}
Let $\gamma_{s}^{t,l}$ be a tagged line segment in $\Omega$, and $Y_{u}$ (resp.  $Y_{v}$) be the counterclockwise neighbor of $Y_{s}$ (resp.  $Y_{t}$). Define the 
\emph{pivot} of 
$\gamma_{s}^{t,l}$ 
 as follows: 

 Case I: $2\leq[s,t]\leq n-2$. In this case, $\gamma_{s}^{t,l}=\gamma_{s}^{t}$.
\begin{itemize}
\item 
If  $\gamma_{s}^{v}$ lies in  
$\Omega$, then $\gamma_{s}^{v}$ is called a pivot of 
$\gamma_{s}^{t}$;
\item
If $\gamma_{u}^{t}$ lies in  $\Omega$, then $\gamma_{u}^{t}$  is called a pivot of 
$\gamma_{s}^{t}$.
\end{itemize}

 Case II: $[s,t]=n-1$. In this case, $\gamma_{s}^{t,l}=\gamma_{s}^{t}$.
\begin{itemize}
\item 
If $t>0$, then  $\gamma_{-t}^{t,1}$ and  $\gamma_{-t}^{t,-1}$ are elements in $\Omega$; therefore, they  are   considered as pivots of $\gamma_{s}^{t}$;  
\item
If $\gamma_{u}^{t}$ lies in  $\Omega$, then
$\gamma_{u}^{t}$ is called a pivot of 
$\gamma_{s}^{t}$.
\end{itemize}

Case III: $[s,t]=n$. In this case, $\gamma_{s}^{t,l}=\gamma_{-t}^{t,\epsilon}$ for $\epsilon=-1 \ {\rm or}\ 1$.
\begin{itemize}
\item 
If $\gamma_{u}^{t}$ lies in $\Omega$, then $\gamma_{u}^{t}$  is called a pivot of
$\gamma_{-t}^{t,\epsilon}$.
\end{itemize}
 \end{definition}
\begin{example}
In Figure~\ref{fig:D}, $\gamma_{-2}^{-1}$ has the three pivots: $\gamma_{-2}^{2,-1}$, $\gamma_{-2}^{2,1}$ and $\gamma_{-4}^{-1}$; $\gamma_{-3}^{2}$ has the two pivots: $\gamma_{-3}^{4}$ and $\gamma_{-1}^{2}$; $\gamma_{-4}^{4,-1}$
has only one pivot: $\gamma_{-3}^{4}$; $\gamma_{2}^{3}$ has no pivots.
\end{example}

 Now  we define a translation quiver $(\Gamma_{D},R_{D})$ with respect to pivots, where $\Gamma_{D}$  is a quiver whose  vertices are all elements in $\Omega$, and  for two tagged line segments $\gamma_{s}^{t,l}, \gamma_{s'}^{t',l'}\in \Omega$, there is an arrow
$\gamma_{s}^{t,l}\to\gamma_{s'}^{t',l'}$   if and only if $\gamma_{s'}^{t',l'}$ is a pivot of $\gamma_{s}^{t,l}$; the translation $R_{D}$ over $\Gamma_{D}$ is defined as follows:
 for a tagged line segment $\gamma_{u}^{v,l}$ in $\Omega$, assume $Y_{s}$ (resp.  $Y_{t}$) is the clockwise neighbor of $Y_{u}$ (resp.  $Y_{v}$), then
\[ 
R_{D}(\gamma_{u}^{v})=
\begin{cases}
\gamma_{s}^{t},     &{\rm if} \  \gamma_{s}^{t}\in\Omega;\\
0, & \text{otherwise}.\\
\end{cases}\ {\rm and}\quad
R_{D}(\gamma_{-v}^{v,\epsilon})=
\begin{cases}
\gamma_{-t}^{t,-\epsilon},   &{\rm if}\  \gamma_{-t}^{t,-\epsilon}\in\Omega;\\
0,   & \text{otherwise}.\\
\end{cases}\]
 Intuitively,  $R_{D}$ acts on tagged line segments by clockwise rotation.

\begin{definition}
 Let $\mathcal{C}_{D}$ be the mesh category of the translation quiver $(\Gamma_{D},R_{D})$. We call $\mathcal{C}_{D}$ the \emph{category of tagged line segments} of $P(Q_{D})$.
\end{definition}

Next, we  construct a functor from  $\mathcal{C}_{D}$ to ${\rm ind}\ Q_{D}$, outlined in three specific steps.
The first step is to define two functors from the category $\mathcal{A}^{'}$ to ${\rm ind}\ Q_{D}$, where $\mathcal{A}^{'}$ is the full subcategory of $\mathcal{A}$ with objects $\{\tau^{-k}P_{i}\in {\rm ind}\ Q_{A}|i=1,2,\dots, n-1\}$.
 \begin{definition}\label{P1,P2}
 Let $\psi_{1}:\mathcal{A}^{'}\to {\rm ind}\ Q_{D}$ be the functor define  on objects by
\[
\psi_{1} \left(\tau^{-k}P_{i}\right)=\begin{cases}
\tau^{-k}\overline{P}_{i}, &{\rm if} \ i=1,2,\dots, n-2; \\
\tau^{-k}\overline{P}_{n-1},     &{\rm if} \ i= n-1 \ {\rm and}\ k\  {\rm mod} \ 2=0;\\
\tau^{-k}\overline{P}_{n},  &{\rm if} \ i= n-1 \ {\rm and}\ k\ {\rm mod}  \ 2=1,\\
\end{cases}
\]
and on  irreducible morphism $ h:M\to N$ by  irreducible morphism $\psi_{1}(h):\psi_{1}(M)\to \psi_{1}(N)$. 
\end{definition}
\begin{definition}
Let $\psi_{2}: \mathcal{A}^{'}\to {\rm ind}\ Q_{D}$ be the functor define  on objects by
\[
\psi_{2} \left(\tau^{-k}P_{i}\right)=\begin{cases}
\tau^{-k}\overline{P}_{i}, &{\rm if} \ i=1,2,\dots, n-2; \\
\tau^{-k}\overline{P}_{n-1},     &{\rm if} \ i= n-1 \ {\rm and}\ k\ {\rm mod}  \ 2=1;\\
\tau^{-k}\overline{P}_{n},  &{\rm if} \ i= n-1 \ {\rm and}\ k\ {\rm mod}  \ 2=0,\\
\end{cases}
\]
and on irreducible morphism $ h:M\to N$ by  irreducible morphism $\psi_{2}(h):\psi_{2}(M)\to \psi_{2}(N)$.
\end{definition}

The second step is to construct two functors from the full subcategories of $\mathcal{C}_{D}$ to the category $\mathcal{C}_{A}$. We start by defining two subsets of $\Omega$. Let 
\[\Omega_{1}=\Omega\setminus\{\gamma_{-t}^{t,-1}\mid  t=1,2,\dots, n-1 \}\ {\rm and}\  \Omega_{2}=\Omega\setminus\{\gamma_{-t}^{t,1}\mid t=1,2,\dots, n-1\}.\]
Set $\mathcal{C}^{1}_{D}$  and $\mathcal{C}^{2}_{D}$  the full subcategory of $\mathcal{C}_{D}$, consisting of objects from 
 $\Omega_{1}$  and $\Omega_{2}$, respectively. 

 Observe that  if  the vertex $X_{h'}$ with $h'\in\{0,1,\dots, 2n-3\}$ of $P(Q_{A})$ is relabeled as $Y_{h}$ of $P(Q_{D})$, then 
an arrow   $\gamma_{u}^{v,l}\to \gamma_{s}^{t,l'}$  exists in $\Gamma_{D}$ if and only if there exists an arrow   $\gamma(u',v')\to \gamma(s',t')$ in $\Gamma_{A}$.
Based on this observation, we can define functors from $\mathcal{C}^{1}_{D}$ and $\mathcal{C}^{2}_{D}$ to $\mathcal{C}_{A}$ as follows:
\begin{definition}\label{f1}
 Let $f_{1}:\mathcal{C}^{1}_{D} \to \mathcal{C}_{A}$ be the functor defined on objects by
\[
f_{1}(\gamma_{s}^{t,l})= \begin{cases}
\gamma(s',t'),     &{\rm if} \ s\ne -t;\\
\gamma(2n-3-t',t'),   & {\rm if} \ s=-t \ {\rm and}\ l=1,\\
\end{cases}
\]
and on pivot $\gamma_{s}^{t,l}\to \gamma_{u}^{v,l'}$  
by pivot $f_{1}(\gamma_{s}^{t,l})\to f_{1}(\gamma_{u}^{v,l'}).$ 
\end{definition}
\begin{definition} \label{f2}
Let $f_{2}:\mathcal{C}^{2}_{D} \to \mathcal{C}_{A}$ be the functor defined on objects by
\[
f_{2}(\gamma_{s}^{t,l})= \begin{cases}
\gamma(s',t'),     &{\rm if} \ s\ne -t;\\
\gamma(2n-3-t',t'),   & {\rm if} \ s=-t \ {\rm and}\ l=-1,\\
\end{cases}
\]
and on pivot $\gamma_{s}^{t,l}\to \gamma_{u}^{v,l'}$ 
by pivot $f_{2}(\gamma_{s}^{t,l})\to f_{2}(\gamma_{u}^{v,l'}).$ 
\end{definition}

In the third step, we will define a functor from the category $\mathcal{C}_{D}$ to ${\rm ind}\ Q_{D}$ using the functors constructed in the previous two steps, together with the equivalent functor $F_{A}: \mathcal{C}_{A}\to {\rm ind}\ Q_{A}$ recalled in Section \ref{subsec.4.1}.  

\begin{proposition}\label{corresponding}
 Assume that the vertex $X_{h'}$ of $P(Q_{A})$ with $ h'\in\{0,1,\dots ,2n-3\}$ is relabeled by $Y_{h}$ of $P(Q_{D})$. For each line segment $\gamma(s',t')\in\omega$, 
\begin{enumerate}
\item [(1)]  there exists $k\in\{0,1,\dots,n-2\}$ such that
$F_{A}(\gamma(s',t'))=\tau^{-k}P_{n-1}$ if and only if  $s'+t'=2n-3$;
 \item [(2)] 
there exist  $k\in\{0,1,\dots,n-2\}$ and $i\in\{1,2,\dots,n-2\}$ such that
$F_{A}(\gamma(s',t'))=\tau^{-k}P_{i}$ if and only if 
 $\gamma(s',t')$ satisfies $2\leq [s,t]\leq n-1$;
 \item [(3)]  
 if there exist $k\in\{0,1,\dots,n-2\}$ and $i\in Q_{A, 0}\setminus \{n-1\}$  such that $F_{A}(\gamma(s',t'))=\tau^{-k}P_{i}$, then 
$
F_{A}(\gamma({2n-3-t'},{2n-3-s'}))=\tau^{-k}P_{2n-2-i}. 
$
\end{enumerate}
\end{proposition}
\begin{proof}
It follows immediately from Remark~\ref{3.2},  Theorem~\ref{thm:FA} and the definition of $F_{A}$.
\end{proof}

According to this proposition, we can obtain the following  conclusion.

\begin{proposition}\label{Ff}
The images of functors $F_{A}\circ f_{1}: \mathcal{C}^{1}_{D}\to {\rm ind}\ Q_{A}$ and $F_{A}\circ f_{2}: \mathcal{C}^{2}_{D}\to {\rm ind}\ Q_{A}$ are both $\mathcal{A}^{'}$. Specifically,
 $$F_{A}(f_{1}(\mathcal{C}^{1}_{D}))=\mathcal{A}^{'}=F_{A}(f_{2}(\mathcal{C}^{2}_{D})).$$
\end{proposition}
\begin{proof}
By the definition of $f_1$ and $f_2$, we get that $f_{1}(\mathcal{C}^{1}_{D})=f_{2}(\mathcal{C}^{2}_{D})$. Thus this statement is proved 
by combining the arguments of Proposition~\ref{corresponding} (1) and (2).
\end{proof}

Therefore, we are ready to define the functor from $\mathcal{C}_{D}$ to ${\rm ind}\ Q_{D}$.

\begin{definition}\label{equivalence}
Let $F_{D}:\mathcal{C}_{D}\to {\rm ind}\ Q_{D}$ be the functor defined on objects by
\[
F_{D}(\gamma_{s}^{t,l})= \begin{cases}
\psi_{1}(F_{A}(f_{1}(\gamma_{s}^{t,l})))     &{\rm if} \ \gamma_{s}^{t,l}\in\Omega_{1};\\
\psi_{2}(F_{A}(f_{2}(\gamma_{-t}^{t,-1})))   & {\rm otherwise},\\
\end{cases}
\]
and on pivot $\gamma_{s}^{t,l}\to \gamma_{u}^{v,l'}$  
by irreducible morphism $F_{D}(\gamma_{s}^{t,l})\to F_{D}(\gamma_{u}^{v,l'}).$ 
\end{definition}
 
Now we are going to  show that ${\rm ind}\ Q_{D}$ 
   is equivalent to $\mathcal{C}_{D}$.

\begin{theorem}\label{Thm:A}
The functor $ F_{D}:\mathcal{C}_{D}\to {\rm ind}\ Q_{D}$ is an equivalence of categories. 
Particularly,
 \begin{enumerate}
\item[(1)]   $F_{D}$ induces bijections
 \[
 \{ \textup{tagged line segments in $\Omega$}\}\to{\rm ind}\  Q_{D};\]
\[\{ \textup{pivots  in $P(Q_{D})$}\} \to \{\textup{irreducible morphisms in $ {\rm ind}\ Q_{D}$}\};
\]
\item[(2)]   $R_{D}$ corresponds to the Auslander-Reiten translation $\tau$ in the following sense
\[F_{D}\circ R_{D}= \tau \circ F_{D};\]
\item[(3)]  $F_{D}$ induces an isomorphism of translation quivers 
 $\left(\Gamma_{D}, R_{D} \right) \to \left(\Gamma_{\mathcal{D}}, \tau \right).$
\end{enumerate}
\end{theorem}
\begin{proof}
 According to the definition of $F_{D}$, it is easy to check that $F_{D}$ is a bijection between the objects of the categories $\mathcal{C}_{D}$ and ${\rm ind}\ Q_{D}$,
and thus a bijection between the vertices of the quivers  $\Gamma_{D}$ and $\Gamma_{\mathcal{D}}$.

Now we consider the relationship between the pivots in $P(Q_{D})$ and irreducible morphisms in ${\rm ind}\ Q_{D}$. 
 By the definition of $f_1$ and $f_2$, $\gamma_2$ is a  pivot of $\gamma_1$ in $P(Q_{D})$ if and only if $f_1(\gamma_2)$ is a pivot of $f_1(\gamma_1)$ or  $f_2(\gamma_2)$ is a pivot of  $f_2(\gamma_1)$ in $P(Q_{A})$. 
 In addition, $\overline{M}\to \overline{N}$ is an irreducible morphism  in $ {\rm ind}\ Q_{D}$
 if and only if either $\psi_{1}^{-1}(\overline{M})\to \psi_{1}^{-1}(\overline{N})$ or $\psi_{2}^{-1}(\overline{M})\to \psi_{2}^{-1}(\overline{N})$
  is an irreducible morphism  in $\mathcal{A}^{'}$,  since $(\Bbbk Q_{A_{2n-3}})G$ is Morita equivalent to $\Bbbk Q_{D}$. 
Therefore, we know from Theorem~\ref{thm:FA} and Proposition~\ref{corresponding} that $F_{D}$ is a bijection between the pivots in $P(Q_{D})$ and irreducible morphisms in $ {\rm ind}\ Q_{D}$. This completes the proof of  statement  (1).

Furthermore, we consider the correspondence between the  Auslander-Reiten translation $\tau$ on $\mathcal{D}$ and the translation  $R_{D}$. We only discuss the case $\gamma_{s}^{t}$ with $2\leq[s,t]\leq n-2$, the others are similar. In this case,  $R_{D}^{-1}(\gamma_{s}^{t})$ is either $0$ or  $\gamma_{u}^{v}\in \Omega$, where $Y_{u}$ (resp. $Y_{v}$) is the counterclockwise neighbor of $Y_{s}$ (resp. $Y_{t}$). 
For $R_{D}^{-1}(\gamma_{s}^{t})=\gamma_{u}^{v}$, we consider the irreducible morphisms $h:\overline{M}\to F_{D}(\gamma_{u}^{v})$ ending at $F_{D}(\gamma_{u}^{v})$. 
Then there exists  $\gamma\in \Omega$ such that $\overline{M}=F_{D}(\gamma)$ since $F_{D}$ is a bijection on objects, which implies that $\gamma\to\gamma_{u}^{v}$ is an arrow in $\Gamma_{D}$. Since $2\leq[u,v]\leq n-2$, by Definition~\ref{pivot}, we get that $\gamma$ is 
either $\gamma_{s}^{v}$ or $\gamma_{u}^{t}$.
If $\gamma=\gamma_{s}^{v}$, then $\gamma_{s}^{t}\to\gamma_{s}^{v}$ is also an arrow in $\Gamma_{D}$(c.f. Figure~\ref{translation}). 
\begin{figure}[h]
\centering
\begin{tikzpicture}[scale=0.7]
\coordinate (B) at (0.5,0);
\fill (B) circle[radius=2pt];
\node (1) at (0,0) {\small $\gamma_{s}^{t}$};
\coordinate (A) at (2,0.8);
\fill (A) circle[radius=2pt];
\node (2) at (2,1.3) {\small$\gamma_{s}^{v}$};
\coordinate (C) at (3.5,0);
\fill (C) circle[radius=2pt];
\node (3) at (4,0) {\small$\gamma_{u}^{v}$};
\draw[->] (0.5+0.2,0.08) -- (1.8,0.72);
\draw[->] (2.2,0.72)--(3.8-0.5,0.05);
\end{tikzpicture}
\hspace{1cm}
\begin{tikzpicture}[scale=0.7]
\coordinate (B) at (0.5,0);
\fill (B) circle[radius=2pt];
\node (1) at (0,0) {\small$\gamma_{s}^{t}$};
\coordinate (A) at (2,-0.8);
\fill (A) circle[radius=2pt];
\node (2) at (2,-1.3) {\small$\gamma_{u}^{t}$};
\coordinate (C) at (3.5,0);
\fill (C) circle[radius=2pt];
\node (3) at (4.0,0) {\small$\gamma_{u}^{v}$};
\draw[->] (0.5+0.2,-0.08) -- (1.8,-0.72);
\draw[->] (2.2,-0.72)--(3.8-0.5,-0.05);
\end{tikzpicture}
\caption{Local graph of quiver $\Gamma_ {D}$}
\label{translation}
\end{figure}
Applying  the functor $F_{D}$ gives an irreducible morphism from $F_{D}(\gamma_{s}^{t})$ to $F_{D}(\gamma_{s}^{v})$.
 If $\gamma=\gamma_{u}^{t}$, then it is analogous to prove that there is an irreducible morphism from $F_{D}(\gamma_{s}^{t})$ to $F_{D}(\gamma_{u}^{t})$.
 Thus, 
according to the definition of Auslander-Reiten translation  $\tau$,  we have \[\tau^{-1}F_{D}(\gamma_{s}^{t})=F_{D}(\gamma_{u}^{v})=F_{D}R_{D}^{-1}(\gamma_{s}^{t}).\] 
For $R_{D}^{-1}(\gamma_{s}^{t})=0$, it follows that  $\tau^{-1}F_{D}(\gamma_{s}^{t})=0$. 
Otherwise, 
by similarly combining the arguments from statement (1) with the  definition of translation quivers, we can deduce that $$ F_{D}R_{D}^{-1}(\gamma_{s}^{t})=\tau^{-1}F_{D}(\gamma_{s}^{t})\ne 0,$$ which contradicts the given condition.
 Therefore, statement (2) and (3) hold. 
\end{proof}
Consequently, we provide a geometric realization of the category ${\rm ind}\ Q_{D}$ via $\mathcal{C}_{D}$. Next, we present an application of this geometric realization: using the equivalent functor $F_{D}$, we can easily  compute   the dimension vectors of all indecomposable $\Bbbk Q_{D}$-modules.
This will provide a great convenience for  illustrating the Auslander-Reiten quiver of $\Bbbk Q_{D}$.

\begin{theorem}\label{dim of ind}
For any  tagged line segment in  $\Omega$, there exist
\begin{equation}  \label{22.2}
\mathbf{dim}\ F_{D}(\gamma_{-t}^{t,l})=\begin{cases}
\sum\limits_{i=n-t}\limits^{n-1}{\mathbf{dim}\ \overline{S}_{i}}, &{\rm if} \  l=1,\\
\sum\limits_{i=n-t}\limits^{n-2}{\mathbf{dim}\ \overline{S}_{i}}+\mathbf{dim}\ \overline{S}_{n}, &{\rm if} \ l=-1; \\
\end{cases}
\end{equation} 
\begin{equation}  \label{22.3}
\mathbf{dim}\ F_{D}(\gamma_{s}^{t})=\begin{cases}
\sum\limits_{i=n+s}\limits^{n+t-1}{\mathbf{dim}\ \overline{S}_{i}}, & {\rm if}\ s<t<0;\\
\sum\limits_{i=n+s}\limits^{n}{\mathbf{dim}\ \overline{S}_{i}}+ \sum\limits_{i=n-t}\limits^{n-2}{\mathbf{dim}\ \overline{S}_{i}}, & {\rm if}\ s<0<t;\\
\sum\limits_{i=n-t}\limits^{n-s-1}{\mathbf{dim}\ \overline{S}_{i}}, & {\rm if}\ 0<s<t,\\
\end{cases}
\end{equation} 
where  $\mathbf{dim}$  denotes  the dimension vector and $\overline{S}_{i}$ is the simple module supported on  $i\in Q_{D,0}$.
\end{theorem}
\begin{proof}
We only discuss the case $s<t<0$, the other cases are similar. In this case, the vertical coordinates of the vertices $Y_s$ and $Y_t$  are both negative.
 Therefore, by rule $  (*)$ and the definition of functor $f_{1}$, it follows that
$$f_{1}(\gamma_{s}^{t})=\gamma(s+n-1,t+n-1).$$
Furthermore, according to the definition of $F_{A}$, we obtain
$$F_{A}f_{1}(\gamma_{s}^{t})=M(s+n,t+n-1).
$$  Follows from Remark~\ref{3.2}, we immediately get that  
\[\mathbf{dim}\ F_{D}(\gamma_{s}^{t})=\mathbf{dim}\ \psi_{1}(M(s+n,t+n-1))=\sum\limits_{i=s+n}\limits^{n+t-1}{\mathbf{dim}\ \overline{S}_{i}}.
\]
Therefore, the statement holds.
\end{proof}
\begin{remark}
The formula (\ref{22.2}) and (\ref{22.3}) of Theorem \ref{dim of ind} may seem complex. Essentially, we  identify the tagged line segments corresponding to the simple modules, and then derive these two formulas through ``vector addition".
\end{remark}

\section{The geometric interpretation of  the  extension group in $\Bbbk Q_{D}\text{-}{\rm mod}$} \label{sec.5}
In this  section,   
we want to determine the dimension of extension group  and short exact sequence between two indecomposable representations over  $Q_{D}$ via the tagged line segments of $P(Q_{D})$. Since $\mathcal{D}$ is a hereditary category, all the $i$-th extension groups, with $i\geq 2$, vanish. Through Serre duality, 
the dimension of  the $0$-th extension groups can be determined by the corresponding  first extension groups. Hence, we only consider the first extension groups.

\begin{definition}\label{positive}
Let $\gamma_1$ and $\gamma_2$ be two tagged line segments  of $P(Q_{D})$.
A common point of    $\gamma_1$ and $\gamma_2$   is called  a
\emph{positive intersection} of $\gamma_{1}$ and $\gamma_{2}$, if it is neither a common starting point nor a common ending point of  $\gamma_1$ and $\gamma_2$,  and the slope of $(\gamma_2)^{\perp}$   is less than that of $(\gamma_1)^{\perp}$, where $\gamma^{\perp}$ denotes the line perpendicular to $\gamma$.
Denote by ${\rm Int}^{+}( \gamma_{1}, \gamma_{2})$ the number of the positive intersections of  $ \gamma_{1}$ and $\gamma_{2}$. 
\begin{figure}[h]
\centering
\begin{tikzpicture}[scale=0.8]
\coordinate (2) at (-1.2,1);
\coordinate (5) at (2.9,1);
\coordinate (7) at (0.1-0.5,0);
\coordinate (9) at (2.6-0.5,0);
\node[right] at (2){\small $\gamma_{1}$};
\node[left]  at (5){\small $\gamma_{2}$};
\draw [->,thick,red](0.1,-0.5) --(2.1,1);
\draw [->,thick,orange] (2.6,-0.5) --(-0.4,1);
\node  at (1,-1){$(a)$};
\end{tikzpicture}
\hspace{0.5cm}
\begin{tikzpicture}[scale=0.8]
\coordinate (2) at (-1.1,1);
\coordinate (5) at (2.0,-0.1);
\coordinate (7) at (0.1-0.5,0);
\coordinate (9) at (2.6-0.5,0);
\node[right] at (2){\small $\gamma_{1}$};
\node  at (5){\small $\gamma_{2}$};
\draw [->,thick,red](-0.4+1.5,-0.5) --(2.6-0.5,0.5);
\draw [->,thick,orange] (2.6-0.5,0.5) --(0.1-0.5,1);
\node  at (1,-1){$(b)$};
\end{tikzpicture}
\hspace{1cm}
\begin{tikzpicture}[scale=0.8]
\coordinate (2) at (1.4,2+0.4-0.5);
\coordinate (5) at (-0.3,0.1 );
\node[below right] at (2){$\gamma_{2}$};
\node at (5){\small $\gamma_{1}$};
\draw [->,thick,red](0.1-0.5,1-0.2) --(1.4,2.3-0.3-0.4);
\draw [->,thick,orange] (0.4,0.5-0.7) --(0.1-0.5,1-0.2);
\node  at (0.6,-0.5){$(c)$};
\end{tikzpicture}
\caption{Positive intersection  of  $\gamma_{1}$ and $\gamma_{2}$}
\label{position}
\end{figure}
 \end{definition}

\begin{remark}\label{4.20}
The number of positive intersections between two tagged line segments is determined  by their position on $P(Q_{D})$, regardless of whether they contain the tag $``|"$.
\end{remark}

Next, we will explore the relationship between  the positive intersections and the non-split short exact sequences   in the category $\mathcal{D}$.  To this end, we require the following lemmas.

\begin{lemma}\label{2+2}
 Let $\eta_{1}$ and $\eta_{2}$ be two short exact sequences of the   following forms:
\[\begin{tikzcd}[ampersand replacement=\&]
\eta_{1}: 0 \arrow[r] \& M_{1} \xrightarrow{{\left(\begin{smallmatrix} \iota_{1}\\ \iota_{2} \end{smallmatrix}\right) }} F \oplus M_{2} \xrightarrow{{\left(\begin{smallmatrix}\nu_{1} & \nu_{2}\end{smallmatrix}\right) }} M_{3}  \arrow[r] \& 0
\end{tikzcd}  
\]
and
\[\eta_{2}: \begin{tikzcd}[ampersand replacement=\&]
0 \arrow[r] \& M_{2}  \xrightarrow{{\left(\begin{smallmatrix} \nu_{2}\\\nu_{3} \end{smallmatrix}\right) }} M_{3} \oplus M_{4} \xrightarrow{{\left(\begin{smallmatrix} \mu_{1} & \mu_{2} \end{smallmatrix}\right) }} M_{5} \arrow[r] \& 0,
\end{tikzcd}
\]  where $F$, $M_i$ are  representations over  $Q$ and  $M_i$ are indecomposable for  all $i\in\{1,2,\dots,5\}$. Then
\[
\begin{tikzcd}[ampersand replacement=\&]
0 \arrow[r] \& M_{1}  \xrightarrow{{\left(\begin{smallmatrix}\iota_{1}\\\nu_{3}\iota_{2} \end{smallmatrix}\right) }}  F \oplus M_{4} \xrightarrow{{\left(\begin{smallmatrix}
\mu_{1}\nu_{1}\  -\mu_{2}\end{smallmatrix}\right) }} M_{5}  \arrow[r] \& 0
\end{tikzcd}
\]
is also a short exact sequence (refer to the right side of Figure \ref{fig:n-1}).
\end{lemma} 
\begin{proof}
Since   $\eta_{2}$ is a short exact sequence,  the sequence
\[\begin{tikzcd}[ampersand replacement=\&]
0 \arrow[r] \& F\oplus M_{2}  \xrightarrow{{\left(\begin{smallmatrix} \nu_{1}& \nu_{2}\\1 & 0\\0 & \nu_{3} \end{smallmatrix}\right) }}  M_{3} \oplus F\oplus  M_{4} \xrightarrow{{\left(\begin{smallmatrix} \mu_{1} & -\mu_{1}\nu_{1} & \mu_{2} \end{smallmatrix}\right) }}  M_{5} \arrow[r] \& 0
\end{tikzcd}
\] is also   exact.
 Consequently, we obtain   the following commutative diagram  with the right-hand square being a pullback:
\[ \begin{tikzcd}[ampersand replacement=\&,column sep=huge]
0 \arrow[r] \& M_{1} \arrow[r, "{{\left(\begin{smallmatrix}\iota_{1}\\\iota_{2} \end{smallmatrix}\right) }} "]  \arrow[d,equal] \&  F\oplus M_{2} \arrow[r, "{{\left(\begin{smallmatrix}\nu_{1} & \nu_{2}\end{smallmatrix}\right) }}"] \arrow[d,"{ {\left(\begin{smallmatrix}1 & 0 \\ 0 &\nu_{3} \end{smallmatrix}\right) }}"] \&M_{3} \arrow[r]\arrow[d,"\mu_{1}"]  \& 0\\
0 \arrow[r] \& M_{1}  \arrow[r, "{{\left(\begin{smallmatrix}\iota_{1}\\\nu_{3}\iota_{2} \end{smallmatrix}\right) }} "] \&  F\oplus M_{4} \arrow[r, "{{\left(\begin{smallmatrix}
\mu_{1}\nu_{1}\  -\mu_{2}\end{smallmatrix}\right) }}"]  \&M_{5} \arrow[r] \& 0
\end{tikzcd}
\] 
The conclusion follows from \cite[Chapter A, Proposition~5.3(a)]{ASS}.
\end{proof}
\begin{figure} 
\centering
\begin{tikzpicture}[scale=0.7,rotate=0]
\coordinate (B) at (0,0);
\fill[orange] (B) circle[radius=2pt];
\node (1) at (-.4,0) {\small{$M_1$}};
\coordinate (A) at (3,1);
\fill [blue](A) circle[radius=2pt];
\node (2) at (3.5,0.7) {\small{$M_3$}};
\coordinate (C) at (2,2);
\fill[blue] (C) circle[radius=2pt];
\node (C) at (1.4,2) {\small{$M_2$}};
\coordinate (D) at (4,2);
\fill[blue] (D) circle[radius=2pt];
\node (D) at (4.6,2) {\small{$M_5$}};
 
\draw[orange,thick] (0,0) -- (3,3);
\draw [orange,thick](0,0) -- (1,-1);
\draw [thick,orange](3,1) -- (1,-1);
\draw [thick,blue](3,1) -- (2,2)--(3,3);
\draw [thick,blue](3,1) -- (4,2)--(3,3);
\coordinate (C) at (3,3);
\fill[blue] (C) circle[radius=2pt];

\node (3) at (3.2,3.4) {\small{$M_4$}};
\node[below] (4) at (1,-1) {\small{${F}$}};
\fill[orange] (1,-1) circle[radius=2pt];
\end{tikzpicture}
\hspace{0.1cm}
\begin{tikzpicture}[scale=0.7,rotate=0]
\coordinate (B) at (0,0);
\fill[orange] (B) circle[radius=2pt];
\node (1) at (-.4,0) {\small{$N$}};
\coordinate (A) at (3,1);
\fill [blue](A) circle[radius=2pt];
\node (2) at (3.6,0.5) {\small{${F_{g}(E)}$}};
\coordinate (C) at (2,2);
\fill[blue] (C) circle[radius=2pt];
\node (C) at (0.7,2) {\small{${\tau^{-h}P_{n-1}}$}};
\coordinate (D) at (4,2);
\fill[blue] (D) circle[radius=2pt];
\node (D) at (5.2,2) {\small{${\tau^{-k}P_{n-1}}$}};
\draw[orange,thick] (0,0) -- (3,3);
\draw [orange,thick](0,0) -- (1,-1);
\draw [thick,orange](3,1) -- (1,-1);
\draw [thick,blue](3,1) -- (2,2)--(3,3);
\draw [thick,blue](3,1) -- (4,2)--(3,3);
\coordinate (C) at (3,3);
\fill[blue] (C) circle[radius=2pt];

\node (3) at (3.2,3.4) {\small{${E}$}};
\node[below] (4) at (1,-1) {\small{${F}$}};
\fill[orange] (1,-1) circle[radius=2pt];
\end{tikzpicture}
\caption{The  short exact sequences in Lemma~\ref{2+2} (left) and 
Lemma~\ref{lem:5.7} (right)}
\label{fig:n-1}
\end{figure}

\begin{remark}\label{2+1}
In Lemma \ref{2+2}, the representation $F$ could be   decomposable, indecomposable or  zero.
\end{remark}

\begin{lemma}\label{lem:n-1}
Let $\tau^{-l}P_{n-1}$ and $\tau^{-k}P_{n-1}$ be two representations over $Q_{A}$ with $0\leq l<k\leq n-2$.
 Then there exist   $m\in\{0,1,\dots,n-2\}$ and $r\in\{1,2,\dots,n-2\}$  such that 
\[\begin{tikzcd}
 0\ar[r] & \tau^{-l}P_{n-1} \ar[r] & \tau^{-m}P_{r}\oplus\tau^{-m}P_{2n-2-r}\ar[r] & \tau^{-k}P_{n-1}\ar[r]  & 0
\end{tikzcd}\]
 is a short exact sequence  in the category $\mathcal{A}$.
\end{lemma}
\begin{proof}
 Since $\mathscr{R}_{\rightarrow}(P_{n-1})=\mathscr{R}_{\leftarrow}(I_{n-1})$ and $I_{n-1}=\tau^{-(n-2)}P_{n-1}$,  
the sets 
$\Sigma_{\rightarrow}(\tau^{-l}P_{n-1})$ and $\Sigma_{\leftarrow}(\tau^{-k}P_{n-1})$ have two points in common. Assume that $E_{1}$ and $E_{2}$ are  indecomposable representations corresponding to these points, respectively. By Proposition \ref{2.2}(3),  we get that
${\rm Ext}_{\mathcal{A}}^{1}(\tau^{-k}P_{n-1},\tau^{-l}P_{n-1})$ is generated by the following 
  non-split  short exact sequence
$$\begin{tikzcd}
\xi: 0\ar[r] & \tau^{-l}P_{n-1} \ar[r] & E_{1}\oplus E_{2}\ar[r] & \tau^{-k}{P}_{n-1}\ar[r]  & 0.
\end{tikzcd}$$
Applying the automorphism $F_{g}$ to $\xi$, we obtain  a short exact sequence
$$\begin{tikzcd}
F_{g}(\xi): 0\ar[r] & \tau^{-l}P_{n-1} \ar[r] & F_{g}(E_{1})\oplus F_{g}(E_{2}) \ar[r] & \tau^{-k}{P}_{n-1}\ar[r]  & 0.
\end{tikzcd}$$
Again by Proposition \ref{2.2}(3),  $F_{g}(E_{1})= E_{2}$.
Thus the proof is completed by equation \eqref{F_{g}}.
\end{proof}

\begin{lemma}\label{lem:5.7}
Let $\xi$ be a non-split short exact sequence  in the category $\mathcal{A}$.
\begin{itemize}
\item[(1)] If $\xi$ is of the form
\[\begin{tikzcd}
\xi: 0\ar[r] & N \xrightarrow{{\left(\begin{smallmatrix}\iota_{1}\\ \sigma_{1} \end{smallmatrix}\right) }} E \oplus F \ar[r] \xrightarrow{{\left(\begin{smallmatrix}\nu_{1}  & \varsigma_{1} \end{smallmatrix}\right) }} \tau^{-k}P_{n-1}\ar[r]  & 0,
\end{tikzcd}
\]
where  $N$, $E$ and $F$  are indecomposable representations  with  $F_g(E)$  belonging to $ \Sigma_{\leftarrow}(\tau^{-k}P_{n-1})$, then there exists      $h\in \{0,1,\dots,k-1\}$ such that  the following short exact sequence holds:
\[
\begin{tikzcd}[cramped, ampersand replacement=\&]
0 \arrow[r] \& N\oplus F_g(N) \arrow[r] \& \tau^{-h}P_{n-1} \oplus F\oplus F_g(F) \arrow[r] \& \tau^{-k}P_{n-1}  \arrow[r] \& 0.
\end{tikzcd}
\]   
\item[(2)] If $\xi$ is of the form
\[\begin{tikzcd}
\xi: 0\ar[r] & \tau^{-k}P_{n-1} \xrightarrow{{\left(\begin{smallmatrix}\iota_{1}\\ \sigma_{1} \end{smallmatrix}\right) }} E \oplus F \ar[r] \xrightarrow{{\left(\begin{smallmatrix}\nu_{1}  & \varsigma_{1} \end{smallmatrix}\right) }} N\ar[r]  & 0,
\end{tikzcd}
\]
where  $N$, $E$ and $F$  are indecomposable representations with   $F_g(E)$  belonging to $ \Sigma_{\rightarrow}(\tau^{-k}P_{n-1})$, then there exists    $h\in \{k+1,\dots,n-2\}$ such that the following short exact sequence holds:
\[
\begin{tikzcd}[cramped, ampersand replacement=\&]
0 \arrow[r] \& \tau^{-k}P_{n-1}\arrow[r] \& \tau^{-h}P_{n-1} \oplus F\oplus F_g(F) \arrow[r] \&  N\oplus F_g(N)   \arrow[r] \& 0.
\end{tikzcd}
\]  
\end{itemize}
\end{lemma}
\begin{proof}
We only prove  part (1); the proof of part (2) follows similarly.  
If $N$ is in the $\tau$-orbit of $P_{n-1}$, then Lemma~\ref{lem:n-1} implies that  statement(1) is valid. Otherwise,
according to  Lemma~\ref{lem:n-1}, 
  there exists  an integer  $h\in \{0,1,\dots,k-1\}$ such that 
\[
\begin{tikzcd}[ampersand replacement=\&]
0 \arrow[r] \& \tau^{-h}P_{n-1}  \xrightarrow{{\left(\begin{smallmatrix}\mu_{1}\\\mu_{2} \end{smallmatrix}\right) }}E \oplus F_{g}(E) \xrightarrow{{\left(\begin{smallmatrix}
-\nu_{1} & \nu_{2} \end{smallmatrix}\right) } } \tau^{-k}P_{n-1}  \arrow[r] \& 0
\end{tikzcd}
\] and
\[
\eta:\begin{tikzcd}[ampersand replacement=\&]
0 \arrow[r] \& N \xrightarrow{{\left(\begin{smallmatrix}\lambda_{1}\\\sigma_{1} \end{smallmatrix}\right) }} \tau^{-h}P_{n-1} \oplus F  \xrightarrow{{\left(\begin{smallmatrix}
\mu_{2} & \varrho_{1} \end{smallmatrix}\right) } } F_{g}(E)\arrow[r] \& 0
\end{tikzcd}
\]
are short  exact sequences, which are depicted on the left side of  Figure~\ref{fig:n-1}.
Since $F_g$  is an exact functor, the following sequences
\[F_g(\xi):\begin{tikzcd}
  0\ar[r] & F_g(N) \xrightarrow{{\left(\begin{smallmatrix}\iota_{2}\\ \sigma_{2} \end{smallmatrix}\right) }} F_g(E) \oplus F_g(F) \ar[r] \xrightarrow{{\left(\begin{smallmatrix}\nu_{2}  & \varsigma_{2} \end{smallmatrix}\right) }} \tau^{-k}P_{n-1}\ar[r]  & 0,
\end{tikzcd}
\]
and
 \[F_g(\eta):
\begin{tikzcd}[ampersand replacement=\&]
0 \arrow[r] \& F_g(N) \xrightarrow{{\left(\begin{smallmatrix}\lambda_{2}\\\sigma_{2} \end{smallmatrix}\right) }} \tau^{-h}P_{n-1} \oplus F_g(F) \xrightarrow{{\left(\begin{smallmatrix}
\mu_{1} & \varrho_{2} \end{smallmatrix}\right) } } E\arrow[r] \& 0
\end{tikzcd}\]
are exact. Based on Proposition~\ref{2.2}(2) and Lemma~\ref{2+2},   we can assume that   $\mu_1\lambda_{1}=\iota_{1}$, $\mu_2\lambda_{2}=\iota_{2}$, $\nu_2\varrho_{1}=\varsigma_{1}$ and $\nu_1\varrho_{2}=\varsigma_{2}$.
Consequently, 
we   deduce   the following commutative diagram  with the right-hand square being a pullback:
\[ \begin{tikzcd}[ampersand replacement=\&, row sep=3em, column sep=3.2em ]
0 \arrow[r] \&N\oplus F_{g}(N)\   \arrow[r, "\setlength{\arraycolsep}{1pt}{{\left(\begin{smallmatrix} \lambda_{1} & \lambda_{2}\\ \sigma_{1} & 0\\ 0 &\sigma_{2}
 \end{smallmatrix}\right) }} "]  \arrow[d,equal] \&  \tau^{-h}P_{n-1}\oplus F\oplus F_g(F)\ \  \arrow[r, " \setlength{\arraycolsep}{1pt}{{\left(\begin{smallmatrix} \nu_{1}\mu_{1} &   \varsigma_{1} &  &\varsigma_{2}
 \end{smallmatrix}\right) }} " ] \arrow[d,"\setlength{\arraycolsep}{1pt}{{\left(\begin{smallmatrix}\mu_{1}&0&\varrho_{2} \\0&1& 0\\ \mu_{2}&\varrho_{1} &0\\0&0&1\end{smallmatrix}\right) }}"] \&\tau^{-k}P_{n-1}  \arrow[d,"\setlength{\arraycolsep}{1pt} {{\left(\begin{smallmatrix}1 \\ 1 \end{smallmatrix}\right) }}"]  \arrow[r] \& 0\\
0 \arrow[r] \&N\oplus F_{g}(N)\  \arrow[r, "\setlength{\arraycolsep}{1pt}{ {\left(\begin{smallmatrix}\iota_{1} & 0\\ \sigma_{1} & 0\\ 0 & \iota_{2} \\ 0 & \sigma_{2} \end{smallmatrix}\right) } } "] \&  E\oplus F \oplus F_{g}(E)\oplus F_g(F)\ \ \arrow[r,  "\setlength{\arraycolsep}{1pt}{ {\left(\begin{smallmatrix}\nu_{1} & \varsigma_{1} &0&0\\0&0& \nu_{2}&\varsigma_{2} \end{smallmatrix}\right) } }" ] \& \tau^{-k}P_{n-1}^{2} \arrow[r]  \& 0
\end{tikzcd}
\]
 This finishes the proof by \cite[Chapter A, Proposition~5.3(a)]{ASS}.
\end{proof}
 Let $\xi$ be an element in ${\rm Ext}_\mathcal{A}^1(X, Y)$. Then it can be represented by a short exact sequence of the form $  0\rightarrow Y\rightarrow E\rightarrow X\rightarrow 0.$  
The relationship between the first extension groups  on the  category $\mathcal{A}$ and the  category of equivariant objects $\mathcal{A}^G$ was established in  \cite{C1} as follows: for a morphism $p\colon X'\rightarrow X$,  the pullback of $\xi$ along $p$ is denoted by $\xi.p$. Similarly, for a morphism
$p'\colon Y\rightarrow Y'$, the pushout of $\xi$ along $p'$ is denoted by $p'.\xi$. For two objects  $(X, \alpha)$ and $(Y, \beta)$ in $\mathcal{A}^G$,  the space ${\rm Ext}^1_\mathcal{A}(X, Y)$ carries a $\Bbbk$-linear $G$-action {associated} to these two objects 
\[ G\times {\rm Ext}^1_\mathcal{A}(X, Y)\to 
{\rm Ext}^1_\mathcal{A}(X, Y), \ (g,\xi )\mapsto \beta_g^{-1}.F_g(\xi).\alpha_g.
\]
Denote by ${\rm Ext}^1_\mathcal{A}(X, Y)^G$  the invariant subspace of  ${\rm Ext}^1_\mathcal{A}(X, Y)$ with the $\Bbbk$-linear $G$-action. 
\begin{lemma}[{See \cite{C1}}]\label{lem:C1}
Let $(X, \alpha)$ and $(Y, \beta)$  be in $\mathcal{A}^G$. Then the forgetful functor $U$ induces a $\Bbbk$-linear isomorphism
\[U: {\rm Ext}^1_{\mathcal{A}^G} ((X, \alpha), (Y, \beta))\to  {\rm Ext}^1_\mathcal{A}(X, Y)^G.
\]
Indeed, for an extension $\xi\colon 0\rightarrow (Y, \beta)\rightarrow (E, \delta)\rightarrow (X, \alpha)\rightarrow 0$ in $\mathcal{A}^G$ the corresponding extension $U(\xi)$ in $\mathcal{A}$  satisfies $\beta_g.F_gU(\xi)=U(\xi).\alpha_g$  by the following commutative diagram
\[\begin{tikzcd}
U(\xi)\colon & 0\ar[r] & Y \ar[d,"\beta_g"] \ar[r] & E\ar[r]\ar[d,"\delta_g"] & X\ar[r] \ar[d,"\alpha_g"] & 0\\
F_gU(\xi)\colon & 0\ar[r] & F_g(Y) \ar[r] & F_g(E)\ar[r] & F_g(X)\ar[r] & 0.
\end{tikzcd}\]
\end{lemma}
Furthermore, we study the properties of the extension groups   
on  $\mathcal{D}$, yielding the following lemma.
\begin{lemma}\label{f:1}
Suppose that $k,l \in \{0,1,\dots,n-2\}$ and $i,j\in\{n-1,n\}$ with  $l<k$. Then 
\[
{\rm dim}_{\Bbbk}{\rm Ext}_{\mathcal{D}}^{1}(\tau^{-k}\overline{P}_{i},\tau^{-l}\overline{P}_{i})=\left\{\begin{array}{ll}
0  &{\rm if} \ \textup{$(k-l)\ {\rm mod}\ 2=0$}; \\
1 &{\rm if} \ \textup{$(k-l)\ {\rm mod}\ 2=1$}.
 \end{array}\right.
\]
If $i\ne j$, then 
\[
{\rm dim}_{\Bbbk}{\rm Ext}_{\mathcal{D}}^{1}(\tau^{-k}\overline{P}_{i},\tau^{-l}\overline{P}_{j})=\left\{\begin{array}{ll}
1  &{\rm if} \ \textup{$(k-l)\ {\rm mod}\ 2=0$}; \\
0 &{\rm if} \ \textup{$(k-l)\ {\rm mod}\ 2=1$}.
 \end{array}\right.
\]
\end{lemma}
\begin{proof}
It is easy to check by using the Auslander-Reiten theory on the category  $\mathcal{D}$.
\end{proof}
 
\begin{proposition}\label{4.28}
Let  $\gamma_{s}^{t,l_{1}},\gamma_{u}^{v,l_{2}}$ be   tagged line segments in $\Omega$.
If   ${\rm Int}^{+}(\gamma_{s}^{t,l_{1}},\gamma_{u}^{v,l_{2}})=1$, then 
 the non-split short exact sequence corresponding to this positive intersection is 
\begin{equation}\label{short exact}
\begin{tikzcd} 
 0\ar[r] & F_{D}(\gamma_{u}^{v,l_{2}}) \ar[r] &  \overline{E}
\ar[r] & F_{D}(\gamma_{s}^{t,l_{1}})  \ar[r]  & 0,
\end{tikzcd}
\end{equation}   where  
\[\overline{E}=\left\{\begin{array}{ll}
F_D(\gamma_{u}^{-u,l_1})\oplus  F_D(\gamma_s^v)  &{\rm if} \ \textup{$s=-t$ and  $u\ne-v$};\\ 
F_D(\gamma_u^t) &{\rm if} \ \textup{$s=-t$, $u=-v$ and $l_{1}\ne l_{2}$}; \\
F_D(\gamma_{-t}^{t,l_2})\oplus   F_D(\gamma_s^v)  &{\rm if} \ \textup{$s\ne-t$ and $u=-v$};\\
F_D(\gamma_u^t)\oplus F_D(\gamma_{-v}^{v,1})\oplus F_D(\gamma_{-v}^{v,-1}) &{\rm if} \ \textup{$s\ne-t$, $u\ne-v$ and $s=-v$}; \\
F_D(\gamma_{-t}^{t,1})\oplus F_D(\gamma_{-t}^{t,-1})\oplus  F_D(\gamma_s^v) &{\rm if} \ \textup{$s\ne-t$, $u\ne-v$ and $u=-t$};\\
F_D(\gamma_u^t)\oplus F_D(\gamma_s^v) &{\rm if} \ \textup{$s\ne-t$, $u\ne-v$, $s\ne-v$ and $u\ne-t$}.
 \end{array}\right.\]
Furthermore, if the endpoint of   $\gamma_{u}^{v,l_{2}}$ (resp. $\gamma_{s}^{t,l_{1}}$) coincides with the starting point of   $\gamma_{s}^{t,l_{1}}$ (resp. $\gamma_{u}^{v,l_{2}}$),  then $\gamma_s^v=0$ (resp. $\gamma_u^t=0$). 
\end{proposition}
\begin{proof} 
Suppose that the vertex $X_{h'}$ of $P(Q_{A})$ with $h'\in\{0,1,\dots ,2n-3\}$ is relabeled  as $Y_{h}$ in $P(Q_{D})$, and denote 
 \begin{align*}
 &\overline{M}:=F_{D}(\gamma_{s}^{t,l_{1}}), \qquad \overline{N}:=\ F_{D}(\gamma_{u}^{v,l_{2}}),\\
 &M:=F_A(\gamma(s',t')), \enspace N:=F_A(\gamma(u',v')).
  \end{align*}
  By   the proof provided in \cite[Theorem~6.8]{BGMR},   there exists  a short exact sequence $\xi$ in  the category 
$\mathcal{A}$:
\[
\begin{tikzcd}\label{short exact sequence}
 \xi:0\ar[r] & N \ar[r] & F_A(\gamma(u',t'))\oplus F_A(\gamma(s',v'))\ar[r] & 
M\ar[r]  & 0.
\end{tikzcd}
\]
We now proceed to prove the statement by considering each case individually.

(1) If $s=-t$ and $u \ne -v$, then we   assume that $M=\tau^{-k}P_{n-1}$ and $F_A(\gamma(u',2n-3-u'))=\tau^{-m}P_{n-1}$ for some $k,m\in\{0,1\dots,n-2\}.$
Again by the proof of \cite[Theorem~6.8]{BGMR},   we have a short exact sequence:  
\[ \begin{tikzcd}
	 0 & {\tau^{-m}P_{n-1}} & {F_A(\gamma(u',t'))\oplus F_A(\gamma(2n-3-t',2n-3-u'))} &   {\tau^{-k}P_{n-1}} & 0.
	\arrow[from=1-1, to=1-2]
	\arrow[from=1-2, to=1-3]
	\arrow[from=1-3, to=1-4]
	\arrow[from=1-4, to=1-5]
\end{tikzcd} \] 
By applying  Lemma~\ref{lem:5.7}(1), we can get the following short exact sequence in $\mathcal{A}$:
\[  
\eta:\begin{tikzcd} 
  0\ar[r] & N\oplus F_{g}(N) \ar[r] &\tau^{-m}P_{n-1}\oplus F_A(\gamma(s',v'))\oplus  F_{g}(F_A(\gamma(s',v')))\ar[r] & 
\tau^{-k}P_{n-1}\ar[r]  & 0.
\end{tikzcd} 
\] 

Since $\mathcal{A}$ is an abelian category, it is a  result from \cite[Section~5.3]{C2} that the category of equivariant objects $\mathcal{A}^{G}$ is also abelian. Consequently, the equivalent functor $\Phi: \mathcal{A}^{G} \to \mathcal{D}$, which operates between abelian categories,  is exact. 
Thanks to  
   Lemma~\ref{lem:C1} and the short exact sequence $\eta$ in $\mathcal{A}$, we obtain the following short exact sequence in $\mathcal{D}$:
\[
\begin{tikzcd}[cramped, ampersand replacement=\&]
0 \arrow[r] \&\overline{N} \arrow[r] \& F_D(\gamma_{u}^{-u,l_{3}})\oplus F_{D}(\gamma_s^v)  \arrow[r] 
\&  F_D(\gamma_{-t}^{t,l_{1}}) \arrow[r] \& 0,
\end{tikzcd}
\] 
where $l_3\in\{-1,1\}$. 
   Then sequence \eqref{short exact} holds, since the space   ${\rm Hom}(F_D(\gamma_{u}^{-u,l_{3}}),F_D(\gamma_{-t}^{t,l_{1}}))\ne 0$ if and only if $l_3=l_1$. 
Furthermore,  we can apply a similar argument to verify that sequence \eqref{short exact} also holds  in the case where  $s\ne-t$  and  $u=-v$. 

(2) If  $s=-t$, $u=-v$, then we   assume that  $N=\tau^{-l}P_{n-1}$ and $M=\tau^{-k}P_{n-1}$ for $0\leq k,l \leq n-2$.  
 Applying the exact functor $\psi$ to $\xi$, we obtain a short exact sequence   in  the category $\mathcal{D}$:
\[
\begin{tikzcd}[column sep=small]
 0\ar[r] &\tau^{-l}\overline{P}_{n-1}\oplus \tau^{-l}\overline{P}_{n} \ar[r] & 
F_D(\gamma_{u}^{t})\oplus F_D(\gamma_{s}^{v})\ar[r] & \tau^{-k}\overline{P}_{n-1}\oplus \tau^{-k}\overline{P}_{n}\ar[r]  & 0.
\end{tikzcd}\]  
It follows from  the condition $l_{1}\ne l_{2}$ and  Lemma~\ref{f:1} that
  $\overline{E}=F_D(\gamma_{u}^{t})$.

(3) If $s\ne-t$, $u\ne-v$, then applying the  exact   functor $\psi$ to   $\xi$ yields the following      exact sequence:
\[
\begin{tikzcd} 
 0\ar[r] & \overline{N} \ar[r] &  \psi(F_A(\gamma(u',t'))\oplus\psi( F_A(\gamma(s',v'))
\ar[r] & \overline{M}  \ar[r]  & 0.
\end{tikzcd}\] 
According to the definition of $F_D$, we get that 
\[\overline{E}=\left\{\begin{array}{ll}
F_D(\gamma_u^t)\oplus F_D(\gamma_s^v) &{\rm if} \ \textup{$s\ne-v$ and $u\ne-t$}; \\
F_D(\gamma_u^t)\oplus F_D(\gamma_{-v}^{v,1})\oplus F_D(\gamma_{-v}^{v,-1}) &{\rm if} \ \textup{$s=-v$  and $u\ne-t$}; \\
F_D(\gamma_{-t}^{t,1})\oplus F_D(\gamma_{-t}^{t,-1})\oplus  F_D(\gamma_s^v) &{\rm if} \ \textup{$s\ne-v$  and $u=-t$}. 
 \end{array}\right.\]
  It is evident that if   $s=-v$ (resp. $u=-t$) then $u\ne -t$ (resp. $s\ne-v$).
This finishes the proof.
\end{proof}

Now we turn to give a geometric interpretation for  the dimension of  first extension group  between two indecomposable representations over  $Q_{D}$.

\begin{definition}\label{crossing number}
Let  $\gamma_{s}^{t,l_{1}}$  and $\gamma_{u}^{v,l_{2}}$ be  tagged line segments in 
$\Omega$.  
 The \emph{crossing number} 
${\rm Int}(\gamma_{s}^{t,l_{1}},\gamma_{u}^{v,l_{2}})$ of $\gamma_{s}^{t,l_{1}}$ and $\gamma_{u}^{v,l_{2}}$ is defined as follows: 
\[{\rm Int}(\gamma_{s}^{t,l_{1}},\gamma_{u}^{v,l_{2}})=\left\{\begin{array}{ll}
{\rm Int}^{+}(\gamma_{-t}^{t,l_{1}},\gamma_{-v}^{v,l_{2}})|\frac{l_{1}-l_{2}}{2}| &{\rm if} \ \textup{ $s=-t$ \ {\rm and}\ $u=-v$}; \\
{\rm Int}^{+}(\gamma_{s}^{t},\gamma_{u}^{v})+{\rm Int}^{+}(\gamma_{-t}^{-s},\gamma_{u}^{v})  &{\rm if} \ \textup{ $s\ne-t$ \ {\rm and}\ $u\ne-v$}; \\
{\rm Int}^{+}(\gamma_{s}^{t,l_{1}},\gamma_{u}^{v,l_{2}}) &  \textup{ otherwise},
 \end{array}\right.\]
 where $|\frac{l_{1}-l_{2}}{2}|$ denotes the absolute value of $\frac{l_{1}-l_{2}}{2}$. We say that  the tagged line segment $\gamma_{s}^{t,l_{1}}$ \emph{crosses}  the  tagged line segment $\gamma_{u}^{v,l_{2}}$ if either ${\rm Int}(\gamma_{s}^{t,l_{1}},\gamma_{u}^{v,l_{2}})$ or ${\rm Int}(\gamma_{u}^{v,l_{2}},\gamma_{s}^{t,l_{1}})$ is non-zero.
\end{definition}

\begin{theorem}\label{E and e}
Let $\overline{M}$ and $\overline{N}$ be  indecomposable representations over the quiver $Q_{D}$. Suppose that $F_{D}(\gamma_{s}^{t,l_{1}})=\overline{M}$ and $F_{D}(\gamma_{u}^{v,l_{2}})=\overline{N}$,  where $\gamma_{s}^{t,l_{1}}$  and $\gamma_{u}^{v,l_{2}}$ are tagged line segments of $P(Q_{D})$.
Then 
\begin{equation}\label{dim=int}
    {\rm dim}_{\Bbbk}{\rm Ext}_{\mathcal{D}}^{1}(\overline{M},\overline{N})={\rm Int}(\gamma_{s}^{t,l_{1}},\gamma_{u}^{v,l_{2}}).
\end{equation}
\end{theorem}
\begin{proof} 
Let the vertex $X_{h'} $  with $h'\in\{0,1,\dots ,2n-3\}$ of $P(Q_{A})$  be relabeled  by $Y_{h}$ of $P(Q_{D})$. Denote  $F_A(\gamma(s',t'))=M$ and  $F_A(\gamma(u',v'))=N$.
 If both ${s}={-t}$ and ${u}={-v}$, then it follows from Definition~\ref{crossing number} and Lemma~\ref{f:1} that equation~\eqref{dim=int} is satisfied.  
 In all other cases, the proof can be divided into the following two cases.

(1) If ${s}\neq {-t}$ and ${u}\neq {-v}$, then  
it follows from Definition~\ref{crossing number} and the proof of \cite[Theorem~6.8]{BGMR} that
$${\rm Int}(\gamma_{s}^{t,l_{1}},\gamma_{u}^{v,l_{2}})={\rm dim}_{\Bbbk}{\rm Ext}_{\mathcal{A}}^{1}(M,N)+{\rm dim}_{\Bbbk} {\rm Ext}_{\mathcal{A}}^{1}(F_{g}(M),N).$$
 Thanks to  the exact functor $\Phi: \mathcal{A}^{G} \to \mathcal{D}$ and   Lemma \ref{lem:C1},   the desired result will follow immediately once we  prove  the validity of equation \eqref{2=1+1}: 
\begin{equation}\label{2=1+1}
{\rm dim}_{\Bbbk}{\rm Ext}_{\mathcal{A}}^{1}(M\oplus F_{g}(M) ,N \oplus F_{g}(N))^{G}={\rm dim}_{\Bbbk}{\rm Ext}_{\mathcal{A}}^{1}(M,N)+{\rm dim}_{\Bbbk} {\rm Ext}_{\mathcal{A}}^{1}(F_{g}(M),N)\end{equation}

If the right-hand side of equation \eqref{2=1+1} equals 0,
then  Proposition~\ref{prop:E} implies that  equation \eqref{2=1+1} holds.

If the right-hand side of equation \eqref{2=1+1} equals 1, then we may assume that ${\rm Ext}_{\mathcal{A}}^{1}(M,N)$ is generated by a non-split short exact sequence $\xi_1$.
Note that an element  of ${\rm Ext}_{\mathcal{A}}^{1}(M\oplus F_{g}(M) ,N \oplus F_{g}(N))$  has the  form  $a \xi_{1} \oplus b F_{g}(\xi_{1}),$ where $a,b\in \Bbbk$. 
Moreover, it is simple to show that
 $a \xi_{1} \oplus b F_{g}(\xi_{1})$ is an element of ${\rm Ext}_{\mathcal{A}}^{1}(M\oplus F_{g}(M) ,N \oplus F_{g}(N))^{G}$ if and only if $a=b$.
 Hence equation \eqref{2=1+1} holds.

If the right-hand side of equation \eqref{2=1+1} equals 2, 
assume that ${\rm Ext}_{\mathcal{A}}^{1}(M,N)$ and ${\rm Ext}_{\mathcal{A}}^{1}(F_{g}(M),N)$  are generated by non-split short exact sequences $\xi_{1}$ and $\xi_{2}$, respectively.
Then, $\xi_{1} \oplus F_{g}(\xi_{1})$ and $\xi_{2} \oplus F_{g}(\xi_{2})$ are elements of ${\rm Ext}_{\mathcal{A}}^{1}(M\oplus F_{g}(M) ,N \oplus F_{g}(N))^{G}$ from  above discussion. 
 Due to the automorphic functor $F_{g}$, we get that
 $\xi_{1} \oplus F_{g}(\xi_{1})$ is $\Bbbk$-linearly independent of $\xi_{2} \oplus F_{g}(\xi_{2})$. 
Given that  ${\rm dim}_{\Bbbk}{\rm Ext}_{\mathcal{D}}^{1}(\overline{M},\overline{N})$ is at most equal to 2,    Lemma \ref{lem:C1} implies that 
equation \eqref{2=1+1} is valid.

(2) Either ${s}={-t}$ or ${u}={-v}$, but not both.  Without loss of generality, we  may assume that
  $\overline{M}=F_{D}(\gamma_{-t}^{t})=\tau^{-k}\overline{P}_{n}$ for some $k\in \{0,1,\dots,n-2\}$.
 An argument similar to the one used in (1) shows that
our problem reduces to
\begin{equation}\label{1=1}
{\rm dim}_{\Bbbk}{\rm Ext}_{\mathcal{A}}^{1}(\tau^{-k}{P}_{n-1} ,N \oplus F_{g}(N))^{G}={\rm dim}_{\Bbbk}{\rm Ext}_{\mathcal{A}}^{1}(\tau^{-k}{P}_{n-1} ,N).\end{equation}

We claim that the left sides of  equation \eqref{1=1} is less than 2. In fact, if ${\rm Ext}_{\mathcal{A}}^{1}(\tau^{-k}P_{n-1},N)$ is generated by a short exact sequence whose middle term is  $E \oplus F$ with    $F_g(E)$  belonging to $ \Sigma_{\leftarrow}(\tau^{-k}P_{n-1})$, then by Lemma~\ref{lem:5.7}(1),
 there exists $h\in\{0,1,\dots,k-1\}$ such that
\[
\xi_{1}:\begin{tikzcd}[cramped, ampersand replacement=\&]
0 \arrow[r] \& N\oplus F_{g}(N) \arrow[r] \& \tau^{-h}P_{n-1}\oplus F\oplus F_{g}(F) \arrow[r] \& \tau^{-k}P_{n-1}  \arrow[r] \& 0
\end{tikzcd}
\] is a  short exact sequence. 
Note that the short exact sequence 
\[
\xi_{2}:\begin{tikzcd}[cramped, ampersand replacement=\&]
0 \arrow[r] \& N\oplus F_{g}(N)  \arrow[r] \& E\oplus F\oplus F_{g}(N)  \arrow[r ] \&\tau^{-k}P_{n-1}  \arrow[r] \& 0
\end{tikzcd}
\]
is $\Bbbk$-linearly independent of $\xi_{1}$. Thus $\{\xi_{1},\xi_{2}\}$ is a  $\Bbbk$-basis of 
${\rm Ext}_{\mathcal{A}}^{1}(\tau^{-k}P_{n-1},N\oplus F_{g}(N))$.  
It  is easy to verified that  ${\rm Ext}_{\mathcal{A}}^{1}(\tau^{-k}P_{n-1},N\oplus F_{g}(N))^{G}$ is generated by $\xi_{1}$.

Since both sides of  equation \eqref{1=1} are less than 2,  it suffices to prove that
 ${\rm Ext}_{\mathcal{A}}^{1}(\tau^{-k}{P}_{n-1} ,N)=0$ if and only if  ${\rm Ext}_{\mathcal{A}}^{1}(\tau^{-k}{P}_{n-1} ,N \oplus F_{g}(N))^{G}=0$.  Thanks to Proposition~\ref{prop:E}, the necessity
 is obvious.
The converse implication is also true.  If not, the first extension group  ${\rm Ext}_{\mathcal{A}}^{1}(\tau^{-k}P_{n-1} ,N)$ is generated by a non-split short exact sequence.
Then Lemma~\ref{lem:5.7}(1) implies that the left-hand side of equation \eqref{1=1} is non-zero.
This leads to a contradiction.
The proof of the theorem is now complete.
\end{proof}

\section{Maximal almost pre-rigid representations over $Q_D$}\label{sec.6}
In this section, we introduce the notion of maximal almost pre-rigid representations over  $Q_D$ and  give a geometric realization of such representations. 
  As an application, we show the endomorphism algebras of these representations are tilted algebra.

Recall  that  a representation is called  \emph{basic} if it  contains no repeated direct summands. We define the  almost pre-rigid  representations over $Q_{D}$.
\begin{definition}\label{def:arr of D}
 A basic representation $\overline{T}$ over  $Q_{D}$  
is called \emph{almost pre-rigid}  provided that  the following  conditions hold: for each pair  $\overline{X}$ and  $\overline{Y}$ of indecomposable summands of $\overline{T}$,
\begin{itemize}
 \item[(1)] if  $\overline{X}= \tau^{-k}\overline{P}_{i}$ and  $\overline{Y}=\tau^{-l}\overline{P}_{j}$  with $i,j\in\{n-1,n\}$, then 
 ${\rm  Ext}_{\mathcal{D}}^{1}(\overline{X},\overline{Y})$ is trivial;
 \item[(2)]  
  otherwise, 
   ${\rm Ext}_{\mathcal{D}}^{1}(\overline{X},\overline{Y})$ is either $0$ or  generated by a short exact sequence 
$$\begin{tikzcd}
 0\ar[r] & \overline{Y} \ar[r] &\overline{E}\ar[r] & \overline{X}\ar[r]  & 0,
\end{tikzcd}$$
where $\overline{E}$ is indecomposable or $\tau^{-h}\overline{P}_{n-1}\oplus\tau^{-h}\overline{P}_{n}$  for some $h\in \{0,1,\dots,n-2\}$. 
\end{itemize}
\end{definition}
\begin{remark}
According to Lemma~\ref{f:1}, it is easy to see that
  if  $ \tau^{-k}\overline{P}_{n-1}\oplus  \tau^{-k}\overline{P}_{n}$ is a direct summand  of  an almost pre-rigid representation  $ \overline{T}$, then   for every $i\in\{n-1,n\}$ and    $l\in\{0,1\dots, n-2\}$,  the representation $\overline{T}\oplus\tau^{-l}\overline{P}_{i}$ fails to be almost pre-rigid. 
\end{remark}

\begin{definition}\label{maximal pre}
An almost pre-rigid representation $\overline{T}$ over  $Q_{D}$ is \emph{maximal almost pre-rigid}, if 
for every non-zero representation $\overline{M}$ over  $Q_{D}$,   $\overline{T}\oplus \overline{M}$ is not almost pre-rigid.
 Let $\overline{mar}(Q_{D})$ denote the set of all maximal almost pre-rigid representations over  $Q_{D}$.
\end{definition}


Now we introduce the definition of tagged triangulation of $P(Q_{D})$ as follows.

\begin{definition}\label{tri of P(Q_{D})}
A \emph{tagged triangulation}  of the punctured $(2n-2)$-gon $P(Q_{D})$  is defined as a maximal  set of  tagged line segments (including the boundary edges) that do not cross in the interior of  $P(Q_{D})$.
\end{definition}
\begin{remark} \label{rek: bi of tri}
In contrast to the  centrally symmetric colored triangulations of a regular $(2n-2)$-gon described in \cite[Proposition~3.16(3)]{MR2031858},  a tagged triangulation   $\mathcal{T}$ of  $P(Q_{D})$ introduced in Definition~\ref{tri of P(Q_{D})}  includes  the boundary edges and  the tagged line segments   $\gamma_{s}^{t}, \gamma_{-t}^{-s}$   in $\mathcal{T}$ satisfy the relation $\gamma_{s}^{t}=\gamma_{-t}^{-s}$.

For each  tagged triangulation   of  $P(Q_{D})$, the  tagged line segments  passing through    $O$  either all share the same tag  or   consist of two tagged line segments connecting the same antipodal points with different tags. 
Consequently, the tagged triangulations of $P(Q_{D})$ are in bijection with the   centrally symmetric colored triangulations of a regular $(2n-2)$-gon.

Furthermore,
 under this bijection, two  centrally symmetric colored triangulations of 
$(2n-2)$-gon are related by one of   the ``type D flips"  defined in  \cite[Proposition~3.16(4)]{MR2031858} if   and only if  their corresponding tagged triangulations $\mathcal{T}_{1},\mathcal{T}_{2}$ of  $P(Q_D)$ satisfy that   there exist two  distinct tagged  line segments $\gamma_1 $ and $\gamma_2$ such that
$\mathcal{T}_{1}\setminus\{\gamma_1\}=\mathcal{T}_{2}\setminus\{\gamma_2\}$.
\end{remark}
 
\begin{example}
We continue with the case of $Q_{D_5}$. Figure~\ref{fig:tri} illustrates four  distinct   tagged triangulations of  $P(Q_{D_5})$. In the first tagged triangulation, there are   two tagged line segments, $\gamma_{-2}^{2,1}$ and $ \gamma_{-2}^{2,-1}$ with different tags connecting  the same antipodal points. Conversely, in the remaining three tagged triangulations, the tagged line segments through   $O$  have the same tag.
\begin{figure}[h]
\centering
\begin{tikzpicture}[scale=1]
\node (1) at (0+12.5,1-6){};
\node (2) at (-0.67+12.5,0.72-6){};
\node (3) at (0.9+12.5,0.44-6){};
\node (4) at (-0.987+12.5,0.16-6){};
\node (5) at (0.987+12.5,-0.16-6){};
\node (6) at (-0.9+12.5,-0.44-6){};
\node (7) at (0.67+12.5,-0.72-6){};
\node (8) at (0+12.5,-1-6){};
\node (9) at (0+12.5,0-6){};
\fill (9) circle(0.07);
\draw[line width=1.2pt,red] (0.45+12.5,0.1-6)--(0.35+12.5,0.3-6);
\draw[->,line width=1.2pt,red]  (-0.9+12.5,-0.44-6)--(0.9+12.5,0.44-6);
\draw[->,line width=1.2pt,red]  (-0.9+12.5,-0.44-6)--(0.9+12.5,0.44-6);
\draw[->,line width=1.5pt,blue,dotted]  (-0.9+12.5,-0.44-6)--(0.9+12.5,0.44-6);
\draw[<-,line width=1.2pt,red] (0+12.5,1-6)--(-0.987+12.5,0.16-6);
\draw[->,line width=1.2pt,red] (0+12.5,-1-6)--(0.987+12.5,-0.16-6);
\draw[->,line width=1.2pt,red] (-0.9+12.5,-0.44-6)--(0.987+12.5,-0.16-6);
\draw[<-,line width=1.2pt,red] (0.9+12.5,0.44-6)--(-0.987+12.5,0.16-6);
\draw (9) node[above right][inner sep=0.75pt] {{\fontsize{0.27em}{0.32em}\selectfont$O$}};
		\draw (1) node[above][inner sep=0.75pt] {{\fontsize{0.27em}{0.32em}\selectfont$Y_{4}$}};
		\draw (2) node[left][inner sep=0.75pt]{{\fontsize{0.27em}{0.32em}\selectfont$Y_{3}$}};
		\draw (3) node[right][inner sep=0.75pt] {{\fontsize{0.27em}{0.32em}\selectfont$Y_{2}$}};
\node (9) at (-1.25+12.5,-0.0-6){{\fontsize{0.27em}{0.32em}\selectfont$Y_{1}$}};
		\draw (5) node[right][inner sep=0.75pt] {{\fontsize{0.27em}{0.32em}\selectfont$Y_{-1}$}};
		\draw (6) node[below left][inner sep=0.75pt] {{\fontsize{0.27em}{0.32em}\selectfont$Y_{-2}$}};
           \draw (7) node[right][inner sep=0.75pt] {{\fontsize{0.27em}{0.32em}\selectfont$Y_{-3}$}};
		\draw (8) node[below][inner sep=0.75pt] {{\fontsize{0.27em}{0.32em}\selectfont$Y_{-4}$}};

\draw[<-,line width=1.2pt,red] (0+12.5,1-6) -- (0.9+12.5,0.44-6);
\draw[->,line width=1.2pt,red] (-0+12.5,-1-6) -- (-0.9+12.5,-0.44-6);
\draw[<-,line width=1.2pt,red] (0.987+12.5,-0.16-6)--(0.67+12.5,-0.72-6);
\draw[->,line width=1.2pt,red] (-0.987+12.5,0.16-6)--(-0.67+12.5,0.72-6);
\draw[->,line width=1.2pt,red]  (-0.67+12.5,0.72-6)--(0+12.5,1-6);
\draw[<-,line width=1.2pt,red]  (0.67+12.5,-0.72-6)--(0+12.5,-1-6);
\draw[->,line width=1.2pt,red]  (-0.9+12.5,-0.44-6)--(-0.987+12.5,0.16-6);
\draw[<-,line width=1.2pt,red] (0.9+12.5,0.44-6)--(0.987+12.5,-0.16-6);
\draw(0+12.5,1-6) -- (0.9+12.5,0.44-6)--(0.987+12.5,-0.16-6)--(0.67+12.5,-0.72-6)--(0+12.5,-1-6)--(-0.9+12.5,-0.44-6)--(-0.987+12.5,0.16-6)--(-0.67+12.5,0.72-6)--(0+12.5,1-6);

\end{tikzpicture}
\hspace{0.1cm}
\begin{tikzpicture}[scale=1]
\node (1) at (0+12.5,1-6){};
\node (2) at (-0.67+12.5,0.72-6){};
\node (3) at (0.9+12.5,0.44-6){};
\node (4) at (-0.987+12.5,0.16-6){};
\node (5) at (0.987+12.5,-0.16-6){};
\node (6) at (-0.9+12.5,-0.44-6){};
\node (7) at (0.67+12.5,-0.72-6){};
\node (8) at (0+12.5,-1-6){};
\node (9) at (0+12.5,0-6){};
\fill (9) circle(0.07);
\draw[->,line width=1.2pt,red]  (0.987+12.5,-0.16-6)--(-0.987+12.5,0.16-6);
\draw[line width=1.2pt,red] (0.45+12.5,0.1-6)--(0.35+12.5,0.3-6);
\draw[line width=1.2pt,red] (-0.45+12.5,-0.05-6)--(-0.4+12.5,0.15-6);
\draw[->,line width=1.2pt,red]  (-0.9+12.5,-0.44-6)--(0.9+12.5,0.44-6);
\draw[->,line width=1.2pt,red]  (-0.9+12.5,-0.44-6)--(0.9+12.5,0.44-6);
\draw[<-,line width=1.2pt,red] (0+12.5,1-6)--(-0.987+12.5,0.16-6);
\draw[->,line width=1.2pt,red] (0+12.5,-1-6)--(0.987+12.5,-0.16-6);
\draw[->,line width=1.2pt,red] (-0.9+12.5,-0.44-6)--(0.987+12.5,-0.16-6);
\draw[<-,line width=1.2pt,red] (0.9+12.5,0.44-6)--(-0.987+12.5,0.16-6);
\draw (9) node[above right][inner sep=0.75pt] {{\fontsize{0.27em}{0.32em}\selectfont$O$}};
		\draw (1) node[above][inner sep=0.75pt] {{\fontsize{0.27em}{0.32em}\selectfont$Y_{4}$}};
		\draw (2) node[left][inner sep=0.75pt]{{\fontsize{0.27em}{0.32em}\selectfont$Y_{3}$}};
		\draw (3) node[right][inner sep=0.75pt] {{\fontsize{0.27em}{0.32em}\selectfont$Y_{2}$}};
\node (9) at (-1.25+12.5,-0.0-6){{\fontsize{0.27em}{0.32em}\selectfont$Y_{1}$}};
		\draw (5) node[right][inner sep=0.75pt] {{\fontsize{0.27em}{0.32em}\selectfont$Y_{-1}$}};
		\draw (6) node[below left][inner sep=0.75pt] {{\fontsize{0.27em}{0.32em}\selectfont$Y_{-2}$}};
           \draw (7) node[right][inner sep=0.75pt] {{\fontsize{0.27em}{0.32em}\selectfont$Y_{-3}$}};
		\draw (8) node[below][inner sep=0.75pt] {{\fontsize{0.27em}{0.32em}\selectfont$Y_{-4}$}};

\draw[->,line width=1.2pt,red]  (-0.9+12.5,-0.44-6)--(0.9+12.5,0.44-6);
\draw[->,line width=1.2pt,red]  (0.987+12.5,-0.16-6)--(-0.987+12.5,0.16-6);
\draw[->,line width=1.2pt,red]  (-0.9+12.5,-0.44-6)--(0.9+12.5,0.44-6);
\draw[<-,line width=1.2pt,red] (0+12.5,1-6) -- (0.9+12.5,0.44-6);
\draw[->,line width=1.2pt,red] (-0+12.5,-1-6) -- (-0.9+12.5,-0.44-6);
\draw[<-,line width=1.2pt,red] (0.987+12.5,-0.16-6)--(0.67+12.5,-0.72-6);
\draw[->,line width=1.2pt,red] (-0.987+12.5,0.16-6)--(-0.67+12.5,0.72-6);
\draw[->,line width=1.2pt,red]  (-0.67+12.5,0.72-6)--(0+12.5,1-6);
\draw[<-,line width=1.2pt,red]  (0.67+12.5,-0.72-6)--(0+12.5,-1-6);
\draw[->,line width=1.2pt,red]  (-0.9+12.5,-0.44-6)--(-0.987+12.5,0.16-6);
\draw[<-,line width=1.2pt,red] (0.9+12.5,0.44-6)--(0.987+12.5,-0.16-6);
\draw(0+12.5,1-6) -- (0.9+12.5,0.44-6)--(0.987+12.5,-0.16-6)--(0.67+12.5,-0.72-6)--(0+12.5,-1-6)--(-0.9+12.5,-0.44-6)--(-0.987+12.5,0.16-6)--(-0.67+12.5,0.72-6)--(0+12.5,1-6);

\end{tikzpicture}
\hspace{0.1cm}
\begin{tikzpicture}[scale=1]
\node (1) at (0+12.5,1-6){};
\node (2) at (-0.67+12.5,0.72-6){};
\node (3) at (0.9+12.5,0.44-6){};
\node (4) at (-0.987+12.5,0.16-6){};
\node (5) at (0.987+12.5,-0.16-6){};
\node (6) at (-0.9+12.5,-0.44-6){};
\node (7) at (0.67+12.5,-0.72-6){};
\node (8) at (0+12.5,-1-6){};
\node (9) at (0+12.5,0-6){};
\fill (9) circle(0.07);
\draw[->,line width=1.2pt,red]  (0.987+12.5,-0.16-6)--(-0.987+12.5,0.16-6);
\draw[->,line width=1.2pt,red]  (-0.9+12.5,-0.44-6)--(0.9+12.5,0.44-6);
\draw[<-,line width=1.2pt,red] (0+12.5,1-6)--(-0.987+12.5,0.16-6);
\draw[->,line width=1.2pt,red] (0+12.5,-1-6)--(0.987+12.5,-0.16-6);
\draw[->,line width=1.2pt,red] (-0.9+12.5,-0.44-6)--(0.987+12.5,-0.16-6);
\draw[<-,line width=1.2pt,red] (0.9+12.5,0.44-6)--(-0.987+12.5,0.16-6);
\draw (9) node[above right][inner sep=0.75pt] {{\fontsize{0.27em}{0.32em}\selectfont$O$}};
		\draw (1) node[above][inner sep=0.75pt] {{\fontsize{0.27em}{0.32em}\selectfont$Y_{4}$}};
		\draw (2) node[left][inner sep=0.75pt]{{\fontsize{0.27em}{0.32em}\selectfont$Y_{3}$}};
		\draw (3) node[right][inner sep=0.75pt] {{\fontsize{0.27em}{0.32em}\selectfont$Y_{2}$}};
\node (9) at (-1.25+12.5,-0.0-6){{\fontsize{0.27em}{0.32em}\selectfont$Y_{1}$}};
		\draw (5) node[right][inner sep=0.75pt] {{\fontsize{0.27em}{0.32em}\selectfont$Y_{-1}$}};
		\draw (6) node[below left][inner sep=0.75pt] {{\fontsize{0.27em}{0.32em}\selectfont$Y_{-2}$}};
           \draw (7) node[right][inner sep=0.75pt] {{\fontsize{0.27em}{0.32em}\selectfont$Y_{-3}$}};
		\draw (8) node[below][inner sep=0.75pt] {{\fontsize{0.27em}{0.32em}\selectfont$Y_{-4}$}};

\draw[->,line width=1.2pt,red]  (-0.9+12.5,-0.44-6)--(0.9+12.5,0.44-6);
\draw[->,line width=1.2pt,red]  (0.987+12.5,-0.16-6)--(-0.987+12.5,0.16-6);
\draw[->,line width=1.2pt,red]  (-0.9+12.5,-0.44-6)--(0.9+12.5,0.44-6);
\draw[<-,line width=1.2pt,red] (0+12.5,1-6) -- (0.9+12.5,0.44-6);
\draw[->,line width=1.2pt,red] (-0+12.5,-1-6) -- (-0.9+12.5,-0.44-6);
\draw[<-,line width=1.2pt,red] (0.987+12.5,-0.16-6)--(0.67+12.5,-0.72-6);
\draw[->,line width=1.2pt,red] (-0.987+12.5,0.16-6)--(-0.67+12.5,0.72-6);
\draw[->,line width=1.2pt,red]  (-0.67+12.5,0.72-6)--(0+12.5,1-6);
\draw[<-,line width=1.2pt,red]  (0.67+12.5,-0.72-6)--(0+12.5,-1-6);
\draw[->,line width=1.2pt,red]  (-0.9+12.5,-0.44-6)--(-0.987+12.5,0.16-6);
\draw[<-,line width=1.2pt,red] (0.9+12.5,0.44-6)--(0.987+12.5,-0.16-6);
\draw(0+12.5,1-6) -- (0.9+12.5,0.44-6)--(0.987+12.5,-0.16-6)--(0.67+12.5,-0.72-6)--(0+12.5,-1-6)--(-0.9+12.5,-0.44-6)--(-0.987+12.5,0.16-6)--(-0.67+12.5,0.72-6)--(0+12.5,1-6);

\end{tikzpicture}
\hspace{0.1cm}
\begin{tikzpicture}[scale=1]
\node (1) at (0+12.5,1-6){};
\node (2) at (-0.67+12.5,0.72-6){};
\node (3) at (0.9+12.5,0.44-6){};
\node (4) at (-0.987+12.5,0.16-6){};
\node (5) at (0.987+12.5,-0.16-6){};
\node (6) at (-0.9+12.5,-0.44-6){};
\node (7) at (0.67+12.5,-0.72-6){};
\node (8) at (0+12.5,-1-6){};
\node (9) at (0+12.5,0-6){};
\fill (9) circle(0.07);
\draw (9) node[above right][inner sep=0.75pt] {{\fontsize{0.27em}{0.32em}\selectfont$O$}};
		\draw (1) node[above][inner sep=0.75pt] {{\fontsize{0.27em}{0.32em}\selectfont$Y_{4}$}};
		\draw (2) node[left][inner sep=0.75pt]{{\fontsize{0.27em}{0.32em}\selectfont$Y_{3}$}};
		\draw (3) node[right][inner sep=0.75pt] {{\fontsize{0.27em}{0.32em}\selectfont$Y_{2}$}};
\node (9) at (-1.25+12.5,-0.0-6){{\fontsize{0.27em}{0.32em}\selectfont$Y_{1}$}};
		\draw (5) node[right][inner sep=0.75pt] {{\fontsize{0.27em}{0.32em}\selectfont$Y_{-1}$}};
		\draw (6) node[below left][inner sep=0.75pt] {{\fontsize{0.27em}{0.32em}\selectfont$Y_{-2}$}};
           \draw (7) node[right][inner sep=0.75pt] {{\fontsize{0.27em}{0.32em}\selectfont$Y_{-3}$}};
		\draw (8) node[below][inner sep=0.75pt] {{\fontsize{0.27em}{0.32em}\selectfont$Y_{-4}$}};
\draw[->,line width=1.2pt,red]  (-0.9+12.5,-0.44-6)--(0.9+12.5,0.44-6);
\draw[->,line width=1.2pt,red]  (0.987+12.5,-0.16-6)--(-0.987+12.5,0.16-6);
\draw[->,line width=1.2pt,red]  (-0.9+12.5,-0.44-6)--(0.9+12.5,0.44-6);
\draw[->,line width=1.2pt,red]  (0+12.5,-1-6)--(0+12.5,1-6);
\draw[<-,line width=1.2pt,red] (0+12.5,1-6) -- (0.9+12.5,0.44-6);
\draw[->,line width=1.2pt,red] (-0+12.5,-1-6) -- (-0.9+12.5,-0.44-6);
\draw[<-,line width=1.2pt,red] (0.987+12.5,-0.16-6)--(0.67+12.5,-0.72-6);
\draw[->,line width=1.2pt,red] (-0.987+12.5,0.16-6)--(-0.67+12.5,0.72-6);
\draw[->,line width=1.2pt,red]  (-0.67+12.5,0.72-6)--(0+12.5,1-6);
\draw[<-,line width=1.2pt,red]  (0.67+12.5,-0.72-6)--(0+12.5,-1-6);
\draw[->,line width=1.2pt,red]  (-0.9+12.5,-0.44-6)--(-0.987+12.5,0.16-6);
\draw[<-,line width=1.2pt,red] (0.9+12.5,0.44-6)--(0.987+12.5,-0.16-6);
\draw(0+12.5,1-6) -- (0.9+12.5,0.44-6)--(0.987+12.5,-0.16-6)--(0.67+12.5,-0.72-6)--(0+12.5,-1-6)--(-0.9+12.5,-0.44-6)--(-0.987+12.5,0.16-6)--(-0.67+12.5,0.72-6)--(0+12.5,1-6);
\draw[->,line width=1.2pt,red] (0.67+12.5,-0.72-6)--(-0.67+12.5,0.72-6);
\draw[line width=1.2pt,red] (0.67+12.5,-0.72-6)--(0+12.5,0-6);
\end{tikzpicture}
\caption{Four  distinct   tagged triangulations of  $P(Q_{D_{5}})$}
\label{fig:tri}
\end{figure}
\end{example}

We have the following  geometric realization for    maximal almost pre-rigid representations over $Q_ {D}$.

\begin{theorem}\label{Thm 1}
The functor $F_{D}$ induces a bijection, also denoted by $F_{D}$
\[{F}_{D}\colon \{ \textup{tagged triangulations of $P(Q_{D})$}\} \to  \overline{mar}(Q_{D}).\]
\end{theorem}
\begin{proof}
Based on  Proposition~\ref{4.28} and Theorem~\ref {E and e}, a  basic representation $\overline{T}$ over  $Q_{D}$ is   almost pre-rigid  if and only if  the   tagged line segments in $F_{D}^{-1}(\overline{T})$ do not
 cross in the interior of  $P(Q_{D})$.
 Therefore, we can conclude that 
 $F_{D}$ induces a bijection  between the set of  tagged triangulations of $P(Q_{D})$ and the set of maximal almost pre-rigid representations over  $Q_{D}$.
\end{proof}
 
\begin{corollary}\label{form}
 Let  $\overline{T}$  denote a maximal almost pre-rigid
   representation over   $Q_{D}$.
\begin{itemize}
\item [(1)] The  representation
$\bigoplus \limits_{m=0}^{n-2} \tau^{-m}\overline{P}_{1}$ is a direct summand of $\overline{T}$.
 \item [(2)]  The number of indecomposable direct summands of $\overline{T}$   is $2n-2$. 
 Additionally, the total number of maximal almost pre-rigid representations over  $Q_{D}$ is given by the  generalized Catalan number $\frac{3n-5}{n-1} \binom{2n-4}{n-2}$.
\end{itemize}
\end{corollary}
\begin{proof}
By Theorem \ref{dim of ind} and the construction of $P(Q_D)$,  $F_{D}^{-1}(\overline{P}_{1})$ is on the boundary of $P(Q_D)$. 
 The statement (1)  is a direct consequence  of Theorem~\ref{Thm:A}(2) and
   Theorem~\ref{Thm 1}, since 
   each tagged triangulation of $P(Q_D)$ includes the boundary edges of $P(Q_D)$.   
Remark~\ref{rek: bi of tri} and \cite[Proposition~3.16(3)]{MR2031858} imply that each  tagged triangulation  of $P(Q_{D})$  contains $n-1$  tagged line segments  in  the interior of $P(Q_{D})$ and the number of  tagged triangulations of $P(Q_{D})$   is given by the   generalized Catalan number $\frac{3n-5}{n-1} \binom{2n-4}{n-2}$.  
 Thus,  the second  statement follows from  Theorem~\ref{Thm 1}.
\end{proof}
 
Based on the aforementioned corollary, we    provide an application of maximal almost pre-rigid representations in tilting theory at the end of this section.
Recall from  \cite{BGMR} that there is a link between $Q_{A}$ and a type $\mathbb{A}$ quiver  $Q_{\overline{A}}$, which contains  $4n-7$ vertices 
 and possesses directional symmetry. Precisely, through replacing each arrow $i\to (i+1)$ in $Q_{A}$ by a path of length two $i\to \frac{2i+1}{2}\to (i+1)$ and replacing each arrow $i\leftarrow (i+1)$ in $Q_{A}$ by a path of length two $i\leftarrow \frac{2i+1}{2}\leftarrow (i+1)$, one can obtain  $Q_{\overline{A}}$. 

 Similarly as \cite{BGMR}, 
for the quiver $Q_{D}$,
we associate $Q_{D}$ with a type $\mathbb{D}$ quiver  $Q_{\overline{D}}$  containing $2n-2$ vertices as follows:  
\begin{itemize}
\item 
replacing arrow $i\xrightarrow{\beta_{i}} (i+1)$  in $Q_{D}$ by a path of length two $i\xrightarrow{\beta(i)} \frac{2i+1}{2} \xrightarrow{\beta(\frac{2i+1}{2})} (i+1)$, for all $i=1,2,\dots,n-3$;
\item  replacing arrow $i\xleftarrow{\beta_{i}} (i+1)$ in $Q_{D}$ by a path of length two  $i\xleftarrow{\beta(i)}  \frac{2i+1}{2}\ \xleftarrow{\beta(\frac{2i+1}{2})} (i+1)$, for all $i=1,2,\dots,n-3$;
\item  
replacing two arrows $(n-1)\xrightarrow{\beta_{n-2}}  (n-2)$ and $n\xrightarrow{\beta_{n-1}}  (n-2)$ in $Q_{D}$ by three arrows $(n-1)\xrightarrow{\beta(\frac{2n-3}{2})}  \frac{2n-3}{2}$, $n\xrightarrow{\beta(\frac{2n-1}{2})}  \frac{2n-3}{2} $  and $ \frac{2n-3}{2}  \xrightarrow{\beta(n-2)}  (n-2)$;
\item  
replacing two arrows  $(n-1)\xleftarrow{\beta_{n-2}} (n-2)$ and $n\xleftarrow{\beta_{n-1}} (n-2)$ in $Q_{D}$ by three arrows $(n-1)\xleftarrow{\beta(\frac{2n-3}{2})}   \frac{2n-3}{2} $, $n\xleftarrow{\beta(\frac{2n-1}{2})} 
  \frac{2n-3}{2} $ and $ \frac{2n-3}{2}  \xleftarrow{\beta(n-2)}  (n-2)$,
\end{itemize}
we can obtain the quiver $Q_{\overline{D}}$. See Figure~\ref{fig:D4-GD4}  as an example. 
\begin{figure}[h]
\centering
\begin{tikzpicture}
\node (0) at (-1,0) {\small$Q_{D_{4}}=$};
\node (1) at (0,1) {\small$3$};
\node (a) at (0,-1) {\small$4$};
\node (2) at (1,0) {\small$2$};
\node (3) at (2,0) {\small$1$};
\draw[<-] (0.2,0.85) -- (0.75,0.15);
\draw[<-] (0.2,-0.85) -- (0.75,-0.15);
\draw[->] (1.25,0) -- (1.75,0);
\draw (1.25,0.55) node [anchor=north west] [align=left] {\small$\beta_{1}$};
\draw (0.5,1) node [anchor=north west] [align=left] {\small$\beta_{2}$};
\draw (0.5,-0.3) node [anchor=north west] [align=left] {\small$\beta_{3}$};
\end{tikzpicture}
\hspace{2cm}
\begin{tikzpicture}
\node (0) at (-1,0) {\small $Q_{\overline{D_{4}}}=$};
\node (1) at (0,1) {\small$3$};
\node (a) at (0,-1) {\small$4$};
\node (2) at (1,0) {\small$\frac{5}{2}$};
\node (3) at (2,0) {\small$2$};
\node (4) at (3,0) {\small$\frac{3}{2}$};
\node (5) at (4,0) {\small$1$};
\draw[<-] (0.2,0.85) -- (0.75,0.15);
\draw[<-] (0.2,-0.85) -- (0.75,-0.15);
\draw[<-] (1.25,0) -- (1.75,0);
\draw[->] (2.25,0) -- (2.75,0);
\draw[->] (3.25,0) -- (3.75,0);
\draw (1.1,0.55) node [anchor=north west] [align=left] {\small$\beta(2)$};
\draw (2.1,0.55) node [anchor=north west] [align=left] {\small$\beta(\frac{3}{2})$};
\draw (3.1,0.55) node [anchor=north west] [align=left] {\small$\beta(1)$};
\draw (0.5,1) node [anchor=north west] [align=left] {\small$\beta(\frac{5}{2})$};
\draw (0.5,-0.3) node [anchor=north west] [align=left] {\small$\beta(\frac{7}{2})$};
\end{tikzpicture}
\caption{From the quiver $Q_{D_{4}}$ to the quiver $Q_{\overline{D}_{4}}$}
\label{fig:D4-GD4}
\end{figure}

\begin{remark}
     If the  group $\overline{G}=\{\overline{e},\overline{g}\}$ acts on $\Bbbk Q_{\overline{A}}$ via $\overline{g}(i) =2n-2-i$ and
$\overline{g}(\alpha_{i}) =\alpha_{2n-3-i}$ for  
$i\in Q_{\overline{A},  0}$ and $\alpha_i\in Q_{\overline{A},  1}$, then the skew group algebra $(\Bbbk Q_{\overline{A}})\overline{G}$ is Morita equivalent  to   path algebra $\Bbbk Q_{\overline{D}}$.
\end{remark}

By investigating the representations over the quiver $Q_ {D}$ and $Q_{\overline {D}}$, we find an additive functor
$G_{D}\colon\Bbbk Q_{ {D}}\text{-mod}\to \Bbbk Q_{\overline{D}}\text{-mod}$, which is determined by the following:
 \begin{itemize}
    \item  On objects. For each  indecomposable object $X=(X_i,\varphi_{\beta_{i}})_{i\in Q_{D,0}, \beta_{i} \in Q_{D,1}}$ in $\Bbbk Q_{ {D}}\text{-mod}$. Define $G_{D}({X})=(\widetilde{X_i}, \widetilde{\varphi}_{\beta(i)})$, where $\widetilde{X_i}$ is given as 
 follows: for $i\in Q_{D,0}$, $\widetilde{X}_{i}={X}_{i}$; for $i\in\{1, \dots,n-3\}$,
 \[
 \widetilde{X}_{\frac{2i+1}{2}}=\left\{\begin{array}{ll}
 0 &{\rm if} \ \textup{ ${X}_{i+1}=0$}, \\
 {X}_{i}&\textup{otherwise,}
 \end{array}\right.
 \]
and  
\[\widetilde{X}_{\frac{2n-3}{2}}=\left\{\begin{array}{lll}
0 &{\rm if} \ \textup{ ${X}_{n-1}=0={X}_{n}$}, \\
{X}_{n-2}&\textup{otherwise.}
\end{array}\right.\]  
And $\widetilde{\varphi}_{\beta(i)}$ is given as 
 follows: 
for $i\in\{1,2,\dots,n-1\}$,
$\widetilde{\varphi}_{\beta(\frac{2i+1}{2})}=\varphi_{\beta_{i}}$; for $i\in\{1,2,\dots,n-3\}$,  \[
\widetilde{\varphi}_{\beta(i)}=\left\{\begin{array}{lll}
0 & {\rm if} \ \textup{ $\varphi_{\beta_{i}}=0$}, \\
 \operatorname{Id}_{X_{i}} &\textup{otherwise,}
\end{array}\right.\]
and  
\[
\widetilde{\varphi}_{\beta(n-2)}=\left\{\begin{array}{lll}
0 & {\rm if} \ \textup{ $\varphi_{\beta_{n-2}}=0=\varphi_{\beta_{n-1}}$}, \\
\operatorname{Id}_{X_{n-2}} &\textup{otherwise.}
\end{array}\right.
\]
 Here,  $\operatorname{Id}_{X_{i}}$ denotes the identity transformation on $X_{i}$ for   $i\in\{1,2,\dots,n-2\}$. 
 
 \item On morphisms. 
  Assume that $\zeta\colon {X}\to {Y}$ is a morphism between indecomposable objects $X$ and $Y$ 
  in  $\Bbbk Q_{ {D}}\text{-mod}$. Define $G_{D}(\zeta)=\widetilde{\zeta}\colon\widetilde{X}\to \widetilde{Y},$  where $\widetilde{\zeta_{i}}$ is given as 
 follows: for $i\in Q_{D,0}$,
   $\widetilde{\zeta_{i}}=\zeta_{i}$; for $i\in\{1,2,\dots,n-3\},$
 \[\widetilde{\zeta}_{\frac{2i+1}{2}}=\left\{\begin{array}{ll}
 0 &{\rm if} \  \zeta_{i+1}= 0, \\
  \zeta_i&  \textup{otherwise},
 \end{array}\right.\] and  
 \[\widetilde{\zeta}_{\frac{2n-3}{2}}=\left\{\begin{array}{ll}
0&{\rm if} \ \textup{ $\zeta_{n-1}=0=\zeta_{n}$}, \\
\zeta_{n-2} &  \textup{otherwise}.  
 \end{array}\right.\]
  \end{itemize}
 
\begin{remark} \label{rem:commutative}
(1) Let $\widetilde{P}_{i}$ be the projective module corresponding to $i\in Q_{\overline{D},\ 0}$. Then 
 \[G_{D}(\tau^{-k}\overline{P}_{i})=\left\{\begin{array}{ll}
\tau^{-2k}\widetilde{P}_{i} &{\rm if} \ \textup{ $i\in\{1,2,\dots,n-2\}$}, \\
 \tau^{-2k}\widetilde{P}_{i}&{\rm if} \ \textup{ $i\in\{n-1,n\}$ and $k\ {\rm mod}\ 2=0$},
 \\
 \tau^{-2k}\widetilde{P}_{n}&{\rm if} \ \textup{ $i=n-1$ and $k\ {\rm mod}\ 2=1$},
 \\
 \tau^{-2k}\widetilde{P}_{n-1}&{\rm if} \ \textup{ $i=n$ and $k\ {\rm mod}\ 2=1$},
 \end{array}\right.\]
where $k\in\{0,1,\dots,n-2\}$.
 An irreducible morphism $\zeta \colon X \to Y$
in $\Bbbk Q_{D}\text{-mod}$ is mapped 
under  $G_D$ to the composition of two irreducible morphisms between the  objects $G_D(X)$ and $G_D(Y)$ in $\Bbbk Q_{\overline{D}}\text{-mod}$.

(2) Denote by $\overline{\psi}$ the exact functor  $(\Bbbk Q_{\overline{A}})\overline{G}\otimes_{\Bbbk Q_{\overline{A}}}{-}:\Bbbk Q_{\overline{A}}\text{-mod}
\to \Bbbk Q_{\overline{D}}\text{-mod}$. 
It is easy to check that $\overline{\psi}\cdot G_{A}=G_{D}\cdot\psi,$ where  
$G_{A}$ is the functor from $\Bbbk Q_{A}\text{-mod}$ to $\Bbbk Q_{\overline{A}}\text{-mod}$ defined in \cite[Section 7]{BGMR}.
\end{remark}

\begin{lemma}\label{fully faithful}
    The functor $G_{D}\colon\Bbbk Q_{D}\text{\rm -mod}\to\Bbbk Q_{\overline{D}}\text{\rm -mod}$ is fully faithful.
\end{lemma}
\begin{proof}
    Let $\zeta\colon {X}\to {Y}$ be the morphism in $\Bbbk Q_{D}\text{\rm -mod}$ such that $G_D(\zeta)=0$. Then  $G_{D}(\zeta)_i=0$ for all $i\in Q_{D,0}$. By definition of $Q_D$, we have $\zeta_i=0$ for all $i\in Q_{D,0}$. Hence $\zeta=0$ and $G_D$ is faithful.

    To show that $G_D$ is full, take a morphism $\widetilde{\zeta}\colon G_{D}(X) \to G_{D}(Y)$. We define $\zeta:X\to Y$ by $\zeta_i=\widetilde{\zeta}$  for all $i\in Q_{D,0}$. It is easy to check that  $\zeta:X\to Y$ is a morphism in $\Bbbk Q_{D}\text{\rm -mod}$ and     $G_{D}(\zeta)=\widetilde{\zeta}$.
\end{proof}

Before stating our main result, we recall that the endomorphism algebras of tilting modules over path algebras are called \emph{tilted algebras} \cite{Tilted}, and  the endomorphism algebras of   cluster-tilting objects in a cluster category 
are \emph{cluster-tilted algebras} \cite{2006Quivers}.
If  $\Lambda$ is a  tilted algebra with  global dimension at most 2, then its trivial extension $\Lambda \ltimes{\rm Ext}_{\Lambda}^{2}(D\Lambda,\Lambda) $ is the corresponding  cluster-tilted algebra,
where $D\Lambda$ is the direct sum of the indecomposable injective $\Lambda$-modules and $\Lambda$ is the direct sum of
the indecomposable projective  $\Lambda$-modules, see \cite{MR2409188}.

 Denote  $\overline{\mathcal{C}}$ the cluster category $\mathcal{D}^{b}(\Bbbk Q_{\overline{D}}\text{-mod})/\tau^{-1}[1]$ of $Q_{\overline{D}}$ and $\iota:\Bbbk Q_{\overline{D}}\text{-mod}\to\overline{\mathcal{C}}$
the functor mapping  a representation to its orbit in the cluster category $\overline{\mathcal{C}}$. We get the following theorem.
\begin{theorem}\label{Thm:D}
Let  $\overline{T}$  be a     maximal almost pre-rigid representation over  $Q_{D}$. Then
  the  endomorphism algebra $C={\rm End}_{\mathcal{D}} \overline{T}$  is a tilted algebra of type $Q_{\overline{D}}$. Moreover,  
there is an isomorphism of algebras 
\[ {\rm End}_{\overline{\mathcal{C}}}  
 \widetilde{G}(\overline{T})\cong C \ltimes{\rm Ext}_{C}^{2}(DC,C),  \]
 where $\widetilde{G}_{D}(\overline{T})$ denotes the object $\iota \circ G_{D}(\overline{T})$.
\end{theorem}
\begin{proof}
Denote
\[
 \overline{T}= \bigoplus\limits_{ i=1}^{n} \bigoplus \limits_{k\in S_i} \tau^{-k}\overline{P}_{i}, \qquad
   \overline{N}= \bigoplus \limits_{i=1}^{n-2} \bigoplus \limits_{k\in S_i} \tau^{-k}\overline{P}_{i},  \qquad \overline{M}=\bigoplus \limits_{i=n-1}^{n} \bigoplus \limits_{k\in S_i} \tau^{-k}\overline{P}_{i}, \]
where  
$S_i$ is the subset of $\{0,1,\dots, n-2\}$ for all $i=1,2,\dots,n$.
 For convenience, we denote  the category 
$\Bbbk Q_{\overline{A}}\text{-mod}$  by $\overline{\mathcal{A}}$ and the category $\Bbbk Q_{\overline{D}}\text{-mod}$ by  $\overline{\mathcal{D}}$. 

Suppose that $\psi(N)=\overline{N}$ with $N\in \mathcal{A}$. Since $\overline{T}$ is   almost pre-rigid, it follows from \cite[Proposition~6.5]{BGMR} that $N\oplus  \tau^{-k}P_{n-1}$ is an almost rigid representation over  $Q_{A}$ for all $k\in S_{n-1}\cup S_{n}$.  Consequently, from    \cite[Theorem~7.3(1)]{BGMR},  we have 
 $${\rm Ext}_{\overline{\mathcal{A}}}^{1}(G_{A}(N\oplus\tau^{-k}P_{n-1}),G_{A}(N\oplus\tau^{-k}P_{n-1}))=0,$$  for all $k\in S_{n-1}\cup S_{n}$. Therefore,  by the exactness of  $\overline{\psi}$ and Remark~\ref{rem:commutative}(2),  we obtain the following:
\[
 \begin{array}{ll}
 &{\rm Ext}_{\overline{\mathcal{D}}}^{1}(G_{D}(\overline{N}\oplus\tau^{-k}\overline{P}_{n-1}\oplus\tau^{-k}\overline{P}_{n}),G_{D}(\overline{N}\oplus\tau^{-k}\overline{P}_{n-1}\oplus\tau^{-k}\overline{P}_{n}))\\
 &\enspace={\rm Ext}_{\overline{\mathcal{D}}}^{1}(G_{D}\psi(N\oplus\tau^{-k}P_{n-1}),G_{D}\psi(N\oplus\tau^{-k}P_{n-1}))\\
  &\enspace={\rm Ext}_{\overline{\mathcal{D}}}^{1}(\overline{\psi} G_{A}(N\oplus\tau^{-k}P_{n-1}),\overline{\psi} G_{A}(N\oplus\tau^{-k}P_{n-1}))\\
 &\enspace=0,
 \end{array}
\]for all $k\in S_{n-1}\cup S_{n}$.
Moreover,  by combining Lemma \ref{f:1} with Remark~\ref{rem:commutative}(1), we get that 
\[ {\rm Ext}_{\overline{\mathcal{D}}}^{1} (G_{D}(\overline{M}),G_{D}(\overline{M}))=0.\]

Since  $\Bbbk Q_{\overline{D}}$ is a hereditary algebra and $Q_{\overline{D}}$ contains $2n-2$ vertices,  it follows from Corollary~\ref{form}(2) that 
 $G_{D}(\overline{T})$ is a tilting module over $\Bbbk Q_{\overline{D}}$.
Given that the  global dimension of 
  algebra $\Bbbk Q_{\overline{D}}$ is   at most 1, the global dimension of  $\overline{C}={\rm End}_{\overline{\mathcal{D}}} G_{D}(\overline{T})$ is at most 2 by \cite[Chapter VI, Theorem 4.2]{ASS}.  Consequently,     $\overline{C}$ is a tilted algebra of type $Q_{\overline{D}}$ and its trivial extension $\overline{C} \ltimes{\rm Ext}_{\overline{C}}^{2}(D\overline{C},\overline{C})$ corresponding to the cluster-tilting object $\widetilde{G}_{D}(\overline{T})$. Lemma~\ref{fully faithful} implies that $C\cong \overline{C}$, and thus
  \[ {\rm End}_{\overline{\mathcal{C}}}  
 \widetilde{G}(\overline{T})\cong \overline{C} \ltimes{\rm Ext}_{\overline{C}}^{2}(D\overline{C},\overline{C})\cong C \ltimes{\rm Ext}_{C}^{2}(DC,C).  \]
This confirms that the results hold.
\end{proof}

\section{Representation-theoretic version of    type-$\mathbb{D}$   Cambrian lattices} \label{sec.7} 
In this section, we define   a  poset structure  on set
 $\overline{mar}(Q_D)$, and prove that  the resulting poset is isomorphic to    the  Cambrian lattice  associated with the quiver $Q_D'$, where $Q_D'$ is   obtained from $Q_{D}$ by removing the vertex labeled with $1$ and the arrow $\beta_{1}$.

  Cambrian lattices is a family of   lattices   on finite Coxeter groups, which 
  depend on a selection of Dynkin diagrams along with an orientation of those diagrams  \cite{rea06}.  
For a  finite  Coxeter group $W$, orientations of the Coxeter diagram correspond to Coxeter elements of $W$(cf.\cite{MR1436533}).   
Given this property,  the author reformulates the definition   and denotes the Cambrian lattice determined by   $c$ as $c$-Cambrian lattice in \cite[Section~5]{MR2318219}. 
Additionally,  for any Coxeter element  $c$ of   $W$,  the   $c$-Cambrian lattice is isomorphic to the $c$-cluster lattice,  as established in \cite[Theorem~8.4]{MR2486939}. 

Similar to \cite[Definition 8.1]{BGMR}, we define the poset structure on $\overline{mar}(Q_D)$. 
 
\begin{definition}
 \label{def poset}
 For $\overline{T}_1,\overline{T}_2\in \overline{mar}(Q_D)$, we say that $\overline{T}_1 $ is   \emph{covered by} $\overline{T}_2$,  denoted as $\overline{T}_1 \prec \overline{T}_2$, if 
there exist indecomposable summands $\overline{M}_i$ of $\overline{T}_i$ such that $\overline{T}_1/\overline{M}_1\cong \overline{T}_2/\overline{M}_2$ and 
there  is a non-split short exact sequence 
$0\to\overline{M}_1 \to\overline{E}\to \overline{M}_2\to 0,$
where  $\overline{E}$ is a summand of $\overline{T}_1/\overline{M}_1$.
\end{definition}

\begin{remark}\label{minimal}
The poset $(\overline{mar}(Q_D),\prec)$,
structured by the covering relations given in  Definition~\ref{def poset}, forms  a lattice. This lattice is an analogue to the  poset of maximal almost rigid representations introduced in \cite{BGMR}. In this lattice, there is a unique minimal (resp.  maximal) element, which is given by:
\[
(\bigoplus \limits_{i=2}^{n}  \overline{P}_{i})\oplus(\bigoplus \limits_{m=0}^{n-2} \tau^{-m}\overline{P}_{1}) \ ({\rm resp.}\  (\bigoplus \limits_{i=2}^{n}  \overline{I}_{i})\oplus(\bigoplus \limits_{m=0}^{n-2} \tau^{-m}\overline{P}_{1})).
\]	
Moreover, this covering relation  corresponds to the standard notion of mutation, and the Hasse diagram  of the lattice is the exchange graph.
\end{remark} 

\begin{definition}\label{def: cover T}
Let  $\mathcal{T}_{1},\mathcal{T}_{2}$ be two tagged triangulations of  $P(Q_D)$.  We say that  $\mathcal{T}_{1}$ is  {\em covered by} $\mathcal{T}_{2}$,  denoted $\mathcal{T}_1 \lessdot \mathcal{T}_2$, if there exist  distinct  two tagged  line segments  $\gamma_1$ and $\gamma_2$ such that
$\mathcal{T}_{1}\setminus\{\gamma_1\}=\mathcal{T}_{2}\setminus\{\gamma_2\}$ and the slope of $(\gamma_1)^{\perp}$   is less than that of $(\gamma_2)^{\perp}$, where $\gamma^{\perp}$ denotes the line perpendicular to $\gamma$.
\end{definition}
 
\begin{proposition}\label{perserves the covering}
The map $F_D:\{ \textup{tagged triangulations of $P(Q_{D})$}\}\to \overline{mar}(Q_D)$  is an isomorphism of lattices from 
$(\{ \textup{tagged triangulations of $P(Q_{D})$}\}, \lessdot )$ to
    $(\overline{mar}(Q_D), \prec)$.
\end{proposition}
\begin{proof}
Let  $\mathcal{T}_{1},\mathcal{T}_{2}$ be two tagged triangulations of  $P(Q_D)$. 
 Assume that there exist two distinct tagged  line segments $\gamma_1$ and $\gamma_2$ such that
$\mathcal{T}_{1}\setminus\{\gamma_1\}=\mathcal{T}_{2}\setminus\{\gamma_2\}$. 
It is easy to see that $\gamma_{1}$  crosses  $\gamma_{2}$ in the interior of $P(Q_D)$. 
Therefore, by Proposition~\ref{4.28}, $\mathcal{T}_1$ is covered by $\mathcal{T}_2$ if and only if $F_D(\mathcal{T}_1)$ is covered by $F_D(\mathcal{T}_2)$.   This implies that the statement holds.
\end{proof}

Next, we show  the  lattice 
$(\{ \textup{tagged triangulations of $P(Q_{D})$}\}, \lessdot )$ coincides with the $c$-cluster lattice, where the Coxeter element $c$ corresponds to the quiver $Q_D'$. Specifically, denote the simple generators of the Coxeter group associated with $Q_{D}'$ as $\{\tau_i| i\in Q_{D,0}\setminus\{1\}\}$.  For each pair of noncommuting simple generators $\tau_i$ and $\tau_j$ of   $c$, the edge $i-j$ of  $Q_{D}'$ is oriented $i\to j$ if and only if $\tau_i$ precedes $\tau_j$ in every reduced word for $c$ (cf.\cite{MR1436533}).

Denote  the simple roots of the Coxeter group corresponding  to $Q_D'$ by $\{\pi_i| i\in Q_{D,0}\setminus\{1\}\}$.  We   establish a correspondence between the tagged  line segments  in the interior of ${P}(Q_{D})$ and the almost positive roots of the Coxeter group for  $Q_D'$. Suppose that $F_D(\delta_i)=\overline{P}_{i}$, where  $\delta_i$ is a tagged line segment of ${P}(Q_{D})$ for $i\in Q_{D,0}\setminus\{1\}$. Following  the  identification strategy as outlined in \cite[Section  6]{MR3441664}, the positive root corresponding to
the  tagged line segment  $\gamma_{-t}^{t,l}$  is given by  $ k_2\pi_2+k_3\pi_3+\dots+k_n \pi_{n}$
with $k_i\in\{0,1\}$, and $k_i=1$   if and only if $\delta_i$ is crossed by  $\gamma_{-t}^{t,l}$, where $l\in\{-1,1\}$ and $t\in\{1,2,\dots, n-2\}$.
Additionally, we identify the  tagged line segment  $\delta_i$ with  the negative simple  root $-\pi_i$ for $i\in Q_{D,0}\setminus\{1\}$, and any other 
tagged line segment $\gamma$ in the interior of ${P}(Q_{D})$ with  the positive root obtained by adding the simple roots associated with the tagged line segments crossed by  $\gamma$.

\begin{remark}\label{rek:tri and cluster}
Based on Remark~\ref{rek: bi of tri},  \cite[Remark~1]{MR3441664} and \cite[Section 5]{MR3441664},  the above construction establishes a bijection  between the tagged triangulations of ${P}(Q_{D})$ and  the $c$-clusters, where   the Coxeter element $c$  corresponds to the quiver $Q_D'$.
Therefore, it follows from  Remark~\ref{minimal} and Proposition~\ref{perserves the covering} that the  lattice $(\{ \textup{tagged triangulations of} \ P(Q_{D}) \}, \lessdot )$ 
is isomorphic to the $c$-cluster lattice.  For more details about $c$-cluster lattice, see reference \cite{MR2486939}.
\end{remark} 

\begin{example}
We again consider the case  of $Q_{D_5}$.  The    quiver $Q_{D_5}'$ is depicted on the left side of  Figure~\ref{fig:la}. 
Then  the quiver $Q_{D_5}'$   determines a
  Coxeter element    $c=\tau_3 \tau_2 \tau_4 \tau_5$.
The bijection between  the tagged  line segments  in the interior of ${P}(Q_{D_{5}'})$ and the almost positive roots of the Coxeter group for  $Q_{D_5}'$ is illustrated in    Table~\ref{table}. Furthermore, the $c$-cluster   corresponding to the tagged triangulation shown  in the leftmost of Figure~\ref{fig:tri} is   $\{-\pi_2,-\pi_3,-\pi_4,-\pi_5\}$. 
\begin{figure}[h]
\centering
\begin{tikzpicture}[scale=1]
\node (0) at (-1,0) {\small $Q_{D_{5'}}=$};
\node (1) at (0,1) {\small$4$};
\node (a) at (0,-1) {\small$5$};
\node (2) at (1,0) {\small$3$};
\node (3) at (2,0) {\small$2$};
\draw (1.1,0.53) node [anchor=north west] [align=left] {\small$\beta_{2}$};
\draw (0.5,1) node [anchor=north west] [align=left] {\small$\beta_{3}$};
\draw (0.5,-0.3) node [anchor=north west] [align=left] {\small$\beta_{4}$};
\draw[<-] (0.2,0.85) -- (0.75,0.15);
\draw[<-] (0.2,-0.85) -- (0.75,-0.15);
\draw[->] (1.25,0) -- (1.75,0);
\end{tikzpicture}
\hspace{1cm}
\begin{tikzpicture}[scale=1]
\node (1) at (0+12.5,1-6){};
\node (2) at (-0.67+12.5,0.72-6){};
\node (3) at (0.9+12.5,0.44-6){};
\node (4) at (-0.987+12.5,0.16-6){};
\node (5) at (0.987+12.5,-0.16-6){};
\node (6) at (-0.9+12.5,-0.44-6){};
\node (7) at (0.67+12.5,-0.72-6){};
\node (8) at (0+12.5,-1-6){};
\node (9) at (0+12.5,0-6){};
\fill (9) circle(0.07);
\draw[line width=1.2pt,red] (0.45+12.5,0.1-6)--(0.35+12.5,0.3-6);
\draw[->,line width=1.2pt,red]  (-0.9+12.5,-0.44-6)--(0.9+12.5,0.44-6);
\draw[->,line width=1.2pt,red]  (-0.9+12.5,-0.44-6)--(0.9+12.5,0.44-6);
\node [blue]( ) at  (-0.4+12.5,1-6-1){\small{$\delta_4$}};
\node [red]( ) at  (0.65+12.5,1-6-0.9){\small{$\delta_5$}};
\draw[->,line width=1.5pt,blue,dotted]  (-0.9+12.5,-0.44-6)--(0.9+12.5,0.44-6);
\draw[<-,line width=1.2pt,purple] (0+12.5,1-6)--(-0.987+12.5,0.16-6);
\draw[->,line width=1.2pt,purple] (0+12.5,-1-6)--(0.987+12.5,-0.16-6);
\node[purple] ( ) at  (0.3+12.5,-1-6+0.35){\small{$\delta_2$}};
\node[purple] ( ) at  (-0.3+12.5,1-6-0.4){\small{$\delta_2$}};
\draw[->,line width=1.2pt,orange] (-0.9+12.5,-0.44-6)--(0.987+12.5,-0.16-6);
\draw[<-,line width=1.2pt,orange] (0.9+12.5,0.44-6)--(-0.987+12.5,0.16-6);
\node[orange] ( ) at  (0.3+12.5,-1-6+0.7){\small{$\delta_3$}};
\node [orange]( ) at  (-0.3+12.5,1-6-0.75){\small{$\delta_3$}};
\draw (9) node[above right][inner sep=0.75pt] {{\fontsize{0.27em}{0.32em}\selectfont$O$}};
		\draw (1) node[above][inner sep=0.75pt] {{\fontsize{0.27em}{0.32em}\selectfont$Y_{4}$}};
		\draw (2) node[left][inner sep=0.75pt]{{\fontsize{0.27em}{0.32em}\selectfont$Y_{3}$}};
		\draw (3) node[right][inner sep=0.75pt] {{\fontsize{0.27em}{0.32em}\selectfont$Y_{2}$}};
\node (9) at (-1.25+12.5,-0.0-6){{\fontsize{0.27em}{0.32em}\selectfont$Y_{1}$}};
		\draw (5) node[right][inner sep=0.75pt] {{\fontsize{0.27em}{0.32em}\selectfont$Y_{-1}$}};
		\draw (6) node[below left][inner sep=0.75pt] {{\fontsize{0.27em}{0.32em}\selectfont$Y_{-2}$}};
           \draw (7) node[right][inner sep=0.75pt] {{\fontsize{0.27em}{0.32em}\selectfont$Y_{-3}$}};
		\draw (8) node[below][inner sep=0.75pt] {{\fontsize{0.27em}{0.32em}\selectfont$Y_{-4}$}};
\draw(0+12.5,1-6) -- (0.9+12.5,0.44-6)--(0.987+12.5,-0.16-6)--(0.67+12.5,-0.72-6)--(0+12.5,-1-6)--(-0.9+12.5,-0.44-6)--(-0.987+12.5,0.16-6)--(-0.67+12.5,0.72-6)--(0+12.5,1-6);
\end{tikzpicture}
\caption{The    quiver $Q_{D_5'}$(left) and tagged line segment $\delta_i$ for $i= 2,3,4,5$ (right)}
\label{fig:la}
\end{figure}
\begin{table}
\centering
\renewcommand{\arraystretch}{1.3}
\caption{Representing almost positive roots of   Coxeter group  associated with $Q_{D_5'}$}
\label{table}
\begin{tabular}{|c|c|}
\hline
tagged line segments & almost positive roots of the Coxeter group \\
\hline
$\gamma_{1}^{4}$ & $-\pi_2$ \\
\hline
$\gamma_{1}^{2}$ & $-\pi_3$ \\
\hline
$\gamma_{-2}^{2,1}$ & $-\pi_4$ \\
\hline
$\gamma_{-2}^{2,1}$ & $-\pi_5$ \\
\hline
$\gamma_{2}^{3}$ & $\pi_2$ \\
\hline
$\gamma_{-2}^{4}$ & $\pi_3$ \\
\hline
$\gamma_{-1}^{1,-1}$ & $\pi_4$ \\
\hline
$\gamma_{-1}^{1,1}$ & $\pi_5$ \\
\hline
$\gamma_{-2}^{3}$ & $\pi_2+\pi_3$ \\
\hline
$\gamma_{-4}^{4,-1}$ & $\pi_3+\pi_4$ \\
\hline
$\gamma_{-4}^{4,1}$ & $\pi_3+\pi_5$ \\
\hline
$\gamma_{-3}^{3,-1}$ & $\pi_2+\pi_3+\pi_4$ \\
\hline
$\gamma_{-3}^{3,1}$ & $\pi_2+\pi_3+\pi_5$ \\
\hline
$\gamma_{-1}^{4}$ & $\pi_3+\pi_4+\pi_5$ \\
\hline
$\gamma_{-1}^{3}$ & $\pi_2+\pi_3+\pi_4+\pi_5$ \\
\hline
$\gamma_{-3}^{4}$ & $\pi_2+2\pi_3+\pi_4+\pi_5$ \\
\hline
\end{tabular}
\end{table} 
\end{example}

The preceding discussion  along with Proposition~\ref{perserves the covering}, paves the way for the main result of this section.

 \begin{theorem}\label{Thm:F}
Let $Q_D'$ be the  type-$\mathbb{D}$ quiver obtained from $Q_{D}$ by removing the vertex labeled with $1$ and the arrow $\beta_{1}$.
Then the   Cambrian lattice coming from  $Q_D'$ is isomorphic to the
lattice $(\overline{mar}(Q_D), \prec)$.
\end{theorem}
\begin{proof}
Let $c$  be a Coxeter element determined by $Q_D'$.
By \cite[Theorem~8.4]{MR2486939}, the   Cambrian lattice coming from  $Q_D'$ is isomorphic to  $c$-cluster lattice. Therefore, the result is a direct consequence of  Proposition~\ref{perserves the covering} and  Remark~\ref{rek:tri and cluster}.
\end{proof}

From now on, we turn to give a   representation-theoretic description of the type-$\mathbb{B}$ Cambrian lattice associated with $Q_{D}$ via maximal almost pre-rigid representations over $Q_D$ containing $\tau^{-k }\overline{P}_{n-1}\oplus \tau^{-k }\overline{P}_{n}$  as   direct summands for  some   $k\in \{0,1,\dots, n-2\}$.   
Let  $S(Q_D)$ denote the set of  these maximal almost pre-rigid representations. We define a partial order on  $S(Q_D)$ as follows.

\begin{definition}\label{cover s}
For $\overline{T}_1,\overline{T}_2\in S(Q_D)$,  we say that
  $\overline{T}_1$ is  {\em covered by}  $\overline{T}_2$  if 
there exist   summands $\overline{M}_i$ of $\overline{T}_i$ such that \[\overline{T}_1/\overline{M}_1\cong \overline{T}_2/\overline{M}_2, \quad  {\rm Ext}_{\mathcal{D}}^{1}(\overline{M}_2,\overline{M}_1)\ne0,\] where $\overline{M}_i$ is either 
 $ \tau^{-k_{i}}\overline{P}_{r_{i}}$ or  $\tau^{-k_{i}}\overline{P}_{n-1}\oplus\tau^{-k_{i}}\overline{P}_{n}$ for some  $k_{i}\in\{0,1,\dots,n-2\},\ r_{i}\in\{1,2,\dots,n-2\}$.
\end{definition}

\begin{corollary}
Let $Q_{B}$ be a quiver with the underlying graph as follows: 
\[
\begin{tikzpicture}
\node (2) at (1.7,1) {\small$n-1$};
\node (3) at (3.5,1) {\small$n-2$};
\node (4) at (5,1) {\dots};
\node (5) at (6,1) {\small$3$};
\node (6) at (7,1) {\small$2$};
\draw (2)-- (3)-- (4) -- (5) -- (6) ;
\draw (2.4,1.4) node [anchor=north west] [align=left] {\small$4$};
\end{tikzpicture}\]
where the direction of the arrow between $i $ and $i+1$ is the same as $\beta_{i}\in Q_{D,1}$ for all $i\in\{2,3,\dots,n-2\}$. 
Then the poset   $S(Q_D)$, equipped with the covering relation defined in Definition~\ref{cover s}, is isomorphic to  the Cambrian lattice  coming from $Q_{B}$.
\end{corollary}
\begin{proof}
The clusters associated with  $Q_{B}$
correspond one-to-one with the centrally symmetric triangulations of a regular $(2n-2)$-gon, as detailed in \cite[Section 3.5]{MR2031858}. Consequently,   the proof of this isomorphism follows a similar argument to that used   for  Theorem~\ref{Thm:F}, which  we  omit  here for the sake of brevity.
\end{proof}

\begin{remark}
 According to Theorem~\ref{Thm 1} and  \cite[Theorem 6.8]{BGMR}, there   exists a bijection between $S(Q_D)$  and the set of
maximal almost rigid representations $T$ over $Q_{A}$ satisfying $F_g(T)=T$. That is, $T$ is of   the following form:
\begin{equation}\label{almost symmetric}
\bigoplus \limits_{i=1}^{2n-4} (\tau^{-m_{i}} P_{r_{i}}\oplus \tau^{-m_{i}} P_{2n-2-r_{i}}) \oplus \tau^{-k}P _{n-1},
\end{equation}
where  $m_{i},k\in\{0,1,\dots,n-2\}$, $r_{i}\in\{1,2,\dots,n-2\}$ 
for all $i\in\{1,2,\dots,2n-4\}$.  Therefore, this class of  maximal almost rigid representations  over $Q_{A}$,  equipped with the following covering relation, also provides a representation-theoretic description of the Cambrian lattice  coming from $Q_{B}$:
For any two  maximal almost rigid representations $T_1,T_2$ of the form   indicated in  \eqref{almost symmetric},  $T_1 $ is  covered by $T_2$  if 
there exist  summands $M_i$ of $T_i$ such that 
\[T_1/M_1\cong T_2/M_2,\quad
  {\rm Ext}_{\mathcal{A}}^{1}( {M}_2, {M}_1)\ne0,\] where   $ {M}_i$ is either   $\tau^{-k_{i}} {P}_{n-1}$ or $\tau^{-k_{i}} {P}_{r_{i}}\oplus \tau^{-k_{i}} {P}_{2n-2-r_{i}}$ for    $k_{i}\in\{0,1,\dots,n-2\}$ and $r_{i}\in\{1,2,\dots,n-2\}$.
\end{remark}  

 \subsection*{Acknowledgments}
The authors would like to thank Shiquan Ruan for helpful discussions. This work was partially supported by the National Natural Science Foundation of China (Grant Nos. 12371040, 12131018 and 12161141001).

\end{document}